\documentclass[12pt, oneside]{amsart}

\usepackage[all,cmtip]{xy}
\usepackage{amssymb}
\usepackage{url}
\usepackage{amscd}
\usepackage[colorlinks,urlcolor=blue]{hyperref}
\usepackage{geometry}
\usepackage{comment}
\usepackage{enumerate}
\usepackage[]{graphicx}
\usepackage{bm}
\theoremstyle{plain}
\newtheorem{theorem}{Theorem}[section]
\newtheorem{corollary}[theorem]{Corollary}
\newtheorem{lemma}[theorem]{Lemma}
\newtheorem{proposition}[theorem]{Proposition}

\theoremstyle{remark}
\newtheorem{remark}{Remark}[section]
\newtheorem{example}[remark]{Example}
\newtheorem{definition}[remark]{Definition}

\begin{document}

%\date{\today} 

\title{On branched minimal immersions of surfaces by first eigenfunctions}

\author{Donato Cianci, Mikhail Karpukhin, and Vladimir Medvedev}

\thanks{The first author was supported by a CRM-Laval Postdoctoral Fellowship while some of this research was conducted.}

\address{University of Michigan, Department of Mathematics, 530 Church Street, Ann Arbor, MI 48109}

\email{cianci@umich.edu}

\address{University of California, Irvine, Department of Mathematics, 340 Rowland Hall,
Irvine, CA 92697-3875 }

\thanks{The second author was partially supported by a Schulich Fellowship}

\email{mkarpukh@uci.edu}

\address{D\'{e}partement de Math\'{e}matiques et de Statistique, Pavillon Andr\'{e}-Aisenstadt, Universit\'{e} de Montr\'{e}al, Montr\'{e}al, QC, Canada, H3C 3J7}

\email{medvedevv@dms.umontreal.ca}

%\thanks{}

%\subjclass{}

\begin{abstract}
It was proved by Montiel and Ros that for each conformal structure on a compact surface there is at most one metric which admits a minimal immersion into some unit sphere by first eigenfunctions. We generalize this theorem to the setting of metrics with conical singularities induced from branched minimal immersions by first eigenfunctions into spheres. Our primary motivation is the fact that metrics realizing maxima of the first non-zero Laplace eigenvalue are induced by minimal branched immersions into spheres. In particular, we show that the properties of such metrics induced from $\mathbb{S}^2$ differ significantly from the properties of those induced from $\mathbb{S}^m$ with $m>2$. 
This feature appears to be novel and needs to be taken into account in the existing proofs of the sharp upper bounds for the first non-zero eigenvalue of the Laplacian on the $2$-torus and the Klein bottle.
In the present paper we address this issue and give a detailed overview of the complete proofs of these upper bounds following the works of Nadirashvili, Jakobson-Nadirashvili-Polterovich, El Soufi-Giacomini-Jazar, Nadirashvili-Sire and Petrides.
%We are motivated by the fact that metrics with conical singularities induced from branched minimal immersions into spheres are maxima for the first non-zero Laplace eigenvalue among metrics with area one. 
%As an application of our results, we give a complete proof of the fact that the first non-zero eigenvalue of the Laplacian on the Klein bottle is maximized by the bipolar Lawson surface $\tau_{3,1}$.  While this result follows from the earlier work of Jakobson-Nadirashvili-Polterovich, El Soufi-Giacomini-Jazar, Nadirashvili-Sire and Petrides, its proof has never been written up in full. 
\end{abstract}

\maketitle

%%%%%%%%%%%%%%%%%%%%%%%%%%%%%%%%%%%%%%%%%%%%%%%%%%%%%%%%%%%%%%%%%%%%%%%%%
% Macros
%%%%%%%%%%%%%%%%%%%%%%%%%%%%%%%%%%%%%%%%%%%%%%%%%%%%%%%%%%%%%%%%%%%%%%%%%

\newcommand{\Ric}{\operatorname{Ric}}
\newcommand\cont{\operatorname{cont}}
\newcommand\diff{\operatorname{diff}}
\newcommand{\myVol}{\operatorname{vol}}
\newcommand{\dvol}{\text{dV}}
\newcommand{\GL}{\operatorname{GL}}
\newcommand{\myO}{\operatorname{O}}
\newcommand{\myP}{\operatorname{P}}
\newcommand{\eye}{\operatorname{Id}}
\newcommand{\myF}{\operatorname{F}}
\newcommand{\vol}{\operatorname{vol}}
\newcommand{\odd}{\operatorname{odd}}
\newcommand{\even}{\operatorname{even}}
\newcommand{\ol}{\overline}
\newcommand{\mye}{\operatorname{E}}
\newcommand{\myo}{\operatorname{o}}
\newcommand{\myt}{\operatorname{t}}
\newcommand{\irr}{\operatorname{Irr}}
\newcommand{\mydiv}{\operatorname{div}}
\newcommand{\re}{\operatorname{Re}}
\newcommand{\im}{\operatorname{Im}}
\newcommand{\can}{\mathrm{can}}
\newcommand{\scal}{\operatorname{scal}}
\newcommand{\spec}{\operatorname{spec}}
\newcommand{\tr}{\operatorname{trace}}
\newcommand{\sgn}{\operatorname{sgn}}
\newcommand{\SL}{\operatorname{SL}}
\newcommand{\myspan}{\operatorname{span}}
\newcommand{\mydet}{\operatorname{det}}
\newcommand{\SO}{\operatorname{SO}}
\newcommand{\SU}{\operatorname{SU}}
\newcommand{\specl}{\operatorname{spec_{\mathcal{L}}}}
\newcommand{\fix}{\operatorname{Fix}}
\newcommand{\id}{\operatorname{id}}
\newcommand{\singsup}{\operatorname{singsupp}}
\newcommand{\wave}{\operatorname{wave}}
\newcommand{\ind}{\operatorname{ind}}
\newcommand{\mynull}{\operatorname{null}}
\newcommand{\floor}[1]{\left \lfloor #1  \right \rfloor}

\newcommand\restr[2]{{% we make the whole thing an ordinary symbol
  \left.\kern-\nulldelimiterspace % automatically resize the bar with \right
  #1 % the function
  \vphantom{\big|} % pretend it's a little taller at normal size
  \right|_{#2} % this is the delimiter
  }}

%%%%%%%%%%%%%%%%%%%%%%%%%%%%%%%%%%%%%%%%%%%%%%%%%%%%%%%%%%%%%%%%%%%%%%%%%

\section{Introduction}
Let $(\Sigma, g)$ denote a closed, connected Riemannian surface where the metric $g$ is induced from a minimal isometric immersion into a round sphere of radius $r$. That is, $\Phi \colon (\Sigma,g) \to (\mathbb{S}_r^n, g_{\mathrm{can}})$ is a minimal isometric immersion. By a well known result of Takahashi \cite[Theorem 3]{takahashi}, the coordinate functions of such minimal immersions $\Phi$ are given by eigenfunctions for the Laplace-Beltrami operator on $(\Sigma, g)$ with corresponding eigenvalue $\frac{2}{r^2}$. However, not all immersions are by {\it first} eigenfunctions. The following theorem shows that each conformal class of $\Sigma$ admits at most one metric induced from an immersion into a sphere by first eigenfunctions:

\begin{theorem}[\cite{MR813585,ilias}]
\label{thm:smooth_rigidity}
For each conformal structure on a compact surface, there exists at most one metric which admits an isometric immersion into some unit sphere by first eigenfunctions. 
\end{theorem} 

In this article we generalize Theorem \ref{thm:smooth_rigidity} to the setting of {\it branched} minimal immersions into round spheres by first eigenfunctions (see Theorem \ref{unique} for a precise statement). Branched minimal immersions are given by smooth maps $\Phi \colon \Sigma \to \mathbb{S}^n$ which are minimal immersions except at finitely many points at which $\Phi$ becomes singular. In this situation, the pullback metric $\Phi^*g_{\mathrm{can}}$ on $\Sigma$ will possess conical singularities at the singular points of $\Phi$. Branched minimal immersions into spheres by first eigenfunctions occur in the study of metrics which maximize the first non-zero Laplace eigenvalue, denoted $\lambda_1$, among all metrics of area one. Indeed, in \cite{matthiesen_preprint}, Matthiesen and Siffert proved that for any closed surface $\Sigma$ there exists a metric $\widehat{g}$ of area one, smooth except for possibly finitely many points which correspond to conical singularities, that maximizes $\lambda_1$ among all other unit-area metrics on $\Sigma$. These maximal metrics are induced from branched minimal immersions into a round sphere by first eigenfunctions and do in general possess conical singularities (see \cite{nayatani}).  Therefore, it is natural to study Theorem \ref{thm:smooth_rigidity} in the context of branched minimal immersions.

A technical difficulty unique to the branched immersion case is that one can have branched minimal immersions by first eigenfunctions whose images are an equatorial $2$-sphere. Indeed, the conclusion of Theorem~\ref{thm:smooth_rigidity} is valid only with the restriction that the image of the branched minimal immersion is not an equatorial $2$-sphere. This restriction indicates that the branched minimal immersions by first eigenfunctions into $\mathbb{S}^2$ are in a way special. Moreover, we show that if a conformal class has a metric induced by a branched minimal immersion by first eigenfunctions to $\mathbb{S}^2$ then it does not have a metric induced by a non-trivial branched minimal immersion by first eigenfunctions to a higher-dimensional sphere. 

%???? later
%Thus, any conformal class $c$ belongs to exactly one of the following $3$ categories,
%\begin{itemize}
%\item[1)] there does not exist $g\in c$ such that $g$ admits a branched minimal immersion to a sphere by first eigenfunctions,
%\item[2)] there exists unique $g\in c$ such that $g$ admits a (linearly full??) branched minimal immersion by first eigenfunctions to $\mathbb{S}^m$ with $m\geq 3$,
%\item[3)] there exists $g\in c$ such that $g$ admits a branched minimal immersion by first eigenfunctions to $\mathbb{S}^2$. 
%\end{itemize}
%????

%However, we also prove that if $\Phi_1 \colon \Sigma \to \mathbb{S}^n$ and $\Phi_2 \colon \Sigma \to \mathbb{S}^m$ are branched minimal immersions by first eigenfunctions and $\Phi_1^*\can$ is conformal to $\Phi_2^*\can$, then if the image of $\Phi_1$ lies in an equatorial $2$-sphere then so does the image of $\Phi_2$ (see Theorem \ref{thm:dichotomy} and Remark \ref{rem:dichotomy} for an interesting application to the study of maximal metrics for $\lambda_1$). 

Theorem \ref{thm:smooth_rigidity} has been applied to help classify certain metrics which maximize $\lambda_1$ (see the discussion in the next section). However, our generalization of Theorem \ref{thm:smooth_rigidity} presents a novel feature that needs to be taken into account in this classification.
%to prove that maximal metrics for $\lambda_1$ are smooth (i.e. do not possess conical singularities) on the $2$-torus and the Klein bottle. The existing proofs of these results seem to have overlooked the fact that Theorem~\ref{thm:smooth_rigidity} does not hold if the image of one of the maps is $\mathbb{S}^2$. 
In the present article we address this issue by proving that there are no branched minimal immersions of a torus or a Klein bottle to $\mathbb{S}^2$.
% the latter is impossible for the $2$-torus and the Klein bottle.
% A complete proof of the latter result does not seem to appear in the literature. 
In order to precisely state our results, we give a more detailed version of the previous discussion.    

\subsection{Maximization of the first eigenvalue on surfaces and minimal immersions}
After fixing a surface $\Sigma$, let $\mathcal{R}(\Sigma)$ be the collection of Riemannian metrics on $\Sigma$. We have the following homothety invariant functional on $\mathcal{R}(\Sigma)$: 
\begin{align*}
\bar\lambda_{1}: \mathcal{R}(\Sigma)  \to \mathbb{R}_{\geq 0};  \hspace{20 pt} \bar\lambda_{1}(\Sigma, \cdot): g \mapsto \lambda_{1}(g) \vol(\Sigma,g),
\end{align*}
where $\lambda_1(g)$ is the first non-zero Laplace-Beltrami eigenvalue of $(\Sigma, g)$ and $\vol(\Sigma,g)$ is the area of $(\Sigma, g)$. 

Using the notion of {\it conformal volume} (see Section \ref{background}), Li and Yau \cite{MR674407} established the following upper bound for $\bar\lambda_1(\Sigma,g)$ when $\Sigma$ is orientable and has genus $\gamma$ (see also \cite{MR577325}): 
\begin{align}\label{Li-Yau}
\bar\lambda_1(\Sigma,g) \leq 8 \pi \left \lfloor \frac{\gamma + 3}{2} \right \rfloor,
\end{align}
where the bracket denotes the integer part of the number inside.
Modifying the ideas of Li and Yau, the second author \cite[Theorem 1]{MR3579963} proved the following upper bound for non-orientable surfaces (of genus $\gamma$):
\begin{align}\label{Karpukhin}
\bar\lambda_1(\Sigma, g) \leq 16  \pi \left \lfloor \frac{\gamma +3}{2} \right \rfloor. 
\end{align}
Here the genus of a non-orientable surface is defined to be the genus of its orientable double cover.

Thus, $\bar\lambda_1(\Sigma,g)$ is bounded above on $\mathcal{R}(\Sigma)$. Naturally, one is interested in finding sharp upper bounds for $\bar\lambda_1(\Sigma, g)$ for a given surface and also characterizing the {\it maximal} metrics.\begin{definition}
Let $\Sigma$ be a closed surface. A metric $g_0$ on $\Sigma$ is said to be {\it maximal} for the functional $\bar\lambda_1(\Sigma,g)$ if
\begin{align*}
\bar\lambda_1(\Sigma,g_0)=\sup_{g\in \mathcal{R}(\Sigma) }\bar\lambda_1(\Sigma,g).
\end{align*}
\end{definition} 
\noindent Throughout, we will denote the value of $\sup_{g\in \mathcal{R}(\Sigma) }\bar\lambda_1(\Sigma,g)$ by $\Lambda_1(\Sigma)$. Additionally we set $\Lambda_1(\Sigma,[g])$ to be $\sup_{g\in [g]}\bar\lambda_1(\Sigma,g)$, where $[g]$ denotes the conformal class of a metric $g$. The following theorem guarantees the existence of a maximal metric on $\Sigma$, modulo finitely many points at which the metric may have conical singularities.  

\begin{theorem}[\cite{matthiesen_preprint}]
\label{lit_existence}
For any closed surface $\Sigma$, there is a metric $g$ on $\Sigma$, smooth away from finitely many conical singularities, achieving $\Lambda_1(\Sigma)$, i.e.
\begin{align*}
\Lambda_1(\Sigma) = \bar\lambda_1(\Sigma,g).
\end{align*}
\end{theorem}
\begin{remark}
\begin{enumerate}[(i)]
\item The proof of Theorem~\ref{lit_existence} uses results of Nadirashvili and Sire \cite{MR3427416} and Petrides \cite{petrides} on the maximization of $\Lambda_1(\Sigma,g)$ in a conformal class.
%Nadirashvili and Sire \cite{MR3427416}, and independently Petrides \cite{petrides}, proved an analog of Theorem \ref{lit_existence} for closed {\it orientable} surfaces. 
\item Nayatani and Shoda \cite{nayatani} recently proved that $\Lambda_1$ is maximized by a metric on the Bolza surface with constant curvature one and six conical singularities (this metric was proposed to be maximal in \cite{MR2202582}). Thus, Theorem \ref{lit_existence} is optimal in regards to the regularity of a maximal metric.  
\end{enumerate}
\end{remark}

As the next theorem shows, these maximal metrics for $\bar\lambda_1$ are induced from branched minimal immersions into round spheres. It was first proved by Nadirashvili in \cite{MR1415764} for the particular case of $\bar\lambda_1$. Later, the theorem was generalized to maximal metrics for higher Laplace eigenvalues in \cite{MR2378458}. As noted in \cite{nadirashvili_preprint}, the theorem also holds for metrics with conical singularities (in part because the variational characterization of $\lambda_1$ is the same whether considering metrics with conical singularities or smooth metrics). Together with Theorem \ref{lit_existence}, Theorem \ref{extremal} is our motivation for studying branched minimal immersions by first eigenfunctions. 

\begin{theorem}[\cite{MR1415764},\cite{MR2378458},\cite{kokarev}]\label{induced}
\label{extremal}
Let $g_0$ be a metric on a closed surface $\Sigma$, possibly with conical singularities. Moreover, suppose that:
\begin{align*}
\Lambda_1(\Sigma) = \bar\lambda_1(\Sigma, g_0).
\end{align*}
Then $g_0$ is induced from a (possibly branched) minimal isometric immersion into a sphere by first eigenfunctions. 
\end{theorem}

We briefly review some results regarding $\bar\lambda_1$-maximal metrics (for results regarding {\it extremal} metrics, see the survey \cite{MR3203194} and the papers \cite{el2000riemannian, karpukhin2013nonmaximality, karpukhin2012spectral, karpukhinspectral, lapointe2008spectral, penskoi2010extremal, penskoi2013extremal, penskoi2015generalized}). By Theorem \ref{induced}, any $\bar\lambda_{1}$-maximal metric is induced by a (possibly branched) minimal immersion into a sphere. Hersch proved in 1970 that $\Lambda_1(\mathbb{S}^2)$ is achieved by any constant curvature metric \cite{MR0292357}. By the work of Li and Yau \cite{MR674407}, $\bar\lambda_1(\mathbb{RP}^2,g)\leq 12\pi$ for any metric $g$ with equality for constant curvature metrics. Indeed, the metric of constant curvature one on $\mathbb{RP}^2$ can be realized as the induced metric from a minimal embedding into $\mathbb{S}^4$ called the Veronese embedding. Since there is only one conformal class of metrics on $\mathbb{RP}^2$, Theorem \ref{thm:smooth_rigidity} shows that $\bar\lambda_1(\mathbb{RP}^2,g)\leq 12\pi$ with equality only if the metric is a constant curvature metric. In \cite{MR1415764}, Nadirashvili proved the existence of maximal metrics on the 2-torus (see also \cite{MR2514484}) and outlined a proof of existence for metrics on the Klein bottle. In the next section we discuss the cases of the 2-torus and the Klein bottle in more detail. Finally, the maximal metric is known for $\Sigma_2$, the orientable surface of genus 2. Nayatani and Shoda proved in~\cite{nayatani} that the metric on the Bolza surface proposed in~\cite{MR2202582} is maximal. As a result, $\Lambda_1(\Sigma_2) = 16\pi$.

%Note that those metrics could have conical singularities.
%However, a detailed proof of the existence of smooth maximal metrics on the Klein bottle does not exist in the literature. 

%Specifically, a complete proof of the existence of a smooth maximal metric on the Klein bottle does not appear in the literature. 
 
\subsection{Main results}

We prove the following generalization of Theorem \ref{thm:smooth_rigidity} to the setting of {\it branched} minimal immersions. 

\begin{theorem}\label{unique}
Let $\Sigma$ be a closed surface endowed with a conformal class $c$. Then $c$ belongs to exactly one of the following categories:
\begin{itemize}
\item[1)] There does not exist $g\in c$ such that $g$ admits a branched minimal immersion to a sphere by first eigenfunctions;
\item[2)] There exists a unique $g\in c$ such that $g$ admits a branched minimal immersion by first eigenfunctions to $\mathbb{S}^m$ whose image is not an equatorial $2$-sphere;
\item[3)] There exists $g\in c$ such that $g$ admits a branched minimal immersion by first eigenfunctions to $\mathbb{S}^2$. In this case any two such immersions differ by a post-composition with a conformal automorphism of $\mathbb{S}^2$.
\end{itemize}
 \end{theorem} 

\begin{remark}
\label{Klein}
If $\Sigma$ is not orientable, then $\Phi(\Sigma)$ can never be an equatorial $2$-sphere. Indeed, this would make $\Phi \colon \Sigma \to \mathbb{S}^2$ a branched cover, which is impossible. In Proposition \ref{Xi} we also prove that if $\Sigma$ is a 2-torus the image of a branched minimal immersion by first eigenfunctions cannot be an equatorial $2$-sphere. Thus, category 3) in Theorem \ref{unique} is not possible in these cases. 
\end{remark}

\begin{remark} 
\label{rem:dichotomy}
Theorem~\ref{unique} allows us to construct an example where a maximal metric for $\bar\lambda_1$ cannot be induced by a branched minimal immersion whose components form a basis in the $\lambda_1$-eigenspace. Indeed, in~\cite{nayatani} the authors showed that on a surface of genus $2$ there exists a family of maximal metrics for $\bar\lambda_1$ induced from a branched minimal immersion to $\mathbb{S}^2$. Moreover, there are metrics in the family such that the multiplicity of the first eigenvalue is equal to $5$. At the same time by Theorem~\ref{unique} such a metric can only be induced by a $3$-dimensional family of eigenfunctions.
%Theorem \ref{thm:dichotomy} shows that in general it is not possible to have a maximal metric for $\lambda_1$ be induced by a subspace of first eigenfunctions of dimension $n$ and the same maximal metric also be induced by a subspace of first eigenfunctions of dimension $m$ with $m \neq n$. Indeed, by \cite{nayatani} the maximal metric for $\lambda_1$ among orientable genus two surfaces is induced from a branched minimal immersion of the Bolza surface into $\mathbb{S}^2$. Thus, the maximal metric is induced by a $3$-dimensional family of first eigenfunctions. By Theorem \ref{thm:dichotomy}, any other family of eigenfunctions which induces the maximal metric on the Bolza surface must also by $3$-dimensional.  
\end{remark}

The following theorem was proved by Nadirashvili in~\cite{MR1415764} for the case of $\mathbb{T}^2$, and the same paper contains an outline of the proof for the Klein bottle. Later, Girouard completed some of the steps of this outline in~\cite{MR2514484}.

\begin{theorem}\label{TSmooth}
 The maximal values $\Lambda_1(\mathbb{T}^2)$ and $\Lambda_{1}(\mathbb{KL})$ are achieved by smooth Riemannian metrics. 
\end{theorem}

Let us discuss the proof of Theorem~\ref{TSmooth} presented in~\cite{MR1415764}. Nadirashvili first shows that Theorem~\ref{lit_existence} holds for the torus, i.e. there exists a maximal metric possibly with conical singularities. Then he applies Theorem~\ref{thm:smooth_rigidity} to conclude that the maximal metric is flat. However, as we see from Theorem~\ref{unique} the conclusion of Theorem~\ref{thm:smooth_rigidity} does not hold for branched immersions and special care is needed if the maximal metric happens to be induced by a branched minimal immersion to $\mathbb{S}^2$. Our contribution to Theorem~\ref{TSmooth} is that we show that there are no branched minimal immersions by first eigenfunctions to $\mathbb{S}^2$ from either the 2-torus or the Klein bottle. While the case of the Klein bottle is elementary, see Remark~\ref{Klein}, additional considerations are required to settle the case of the 2-torus, see Proposition~\ref{Xi}.

Once Theorem~\ref{TSmooth} is proved, Nadirashvili's argument shows that the maximal metric on $\mathbb{T}^2$ is flat. Let us recall it in more detail. It follows by Theorem \ref{thm:smooth_rigidity} that any conformal transformation of a smooth maximal metric is an isometry. Since any metric on the 2-torus has a transitive group of conformal transformations, then any smooth maximal metric must have a transitive group of isometries and is therefore flat. It follows that a smooth maximal metric is a scalar multiple of the flat metric on the equilateral torus. In the same paper~\cite{MR1415764}, Nadirashvili used a similar argument to deduce that any smooth maximal metric on the Klein bottle must be a surface of revolution. Later, Jakobson, Nadirashvili, and I. Polterovich \cite{MR2209284} found a candidate for a smooth maximal metric for the Klein bottle (by proving the existence of a metric of revolution that was {\it extremal} for $\bar\lambda_1$). The metric they found corresponded to a bipolar surface of Lawson's $\tau_{3,1}$-torus. Then, in \cite{MR2259925}, El Soufi, Giacomini, and Jazar proved that this metric was the {\it only} smooth extremal metric on the Klein bottle.

\begin{remark}
Note that Theorem \ref{TSmooth} can be obtained by combining Theorem~\ref{unique} and Proposition~\ref{Xi} with the recent work of Matthiesen and Siffert \cite{matthiesen_preprint}. For completeness, we prove Theorem \ref{TSmooth} without the result of Matthiesen and Siffert and instead combine results of Girouard~\cite{MR2514484} with the fact that the functional $\Lambda_1(\Sigma,[g])$ is continuous on the moduli space of conformal classes of metrics on $\Sigma$. The latter fact seems to be well-known but we were unable to find a reference. We present its proof in Section~\ref{Klein_bottle}.
\end{remark}

%Combining the results above with \cite{MR1415764} and \cite{MR2514484} (in the 2-torus case) and   \cite[Theorem 1.3.1]{MR2209284} and \cite[Theorem 1.2]{MR2259925} (in the Klein bottle case), we immediately obtain the following corollary:

As a result of the previous discussion we have the following corollaries of Theorem~\ref{TSmooth}.

\begin{corollary}[\cite{MR1415764}]
%\label{Corollary}
 The maximum for the functional  $\bar\lambda_1(\mathbb{T}^2,g)$ on the space of Riemannian metrics on a 2-torus $\mathbb{T}^2$ is attained if and only if the metric $g$ is homothetic to the flat metric $g_{eq}$ on the equilateral torus and has the following value:

\begin{align*}
\Lambda_1(\mathbb{T}^2)=8\pi^2/\sqrt{3}.
\end{align*}

\end{corollary}

\begin{corollary}[\cite{MR2209284, MR2259925}]
The maximum for the functional  $\bar\lambda_1(\mathbb{KL},g)$ on the space of Riemannian metrics on a Klein bottle $\mathbb{KL}$ is attained if and only if the metric $g$ is homothetic to a metric of revolution:
\begin{align*}
g_0=\frac{9 + (1 + 8 \cos^2 v)^2}{1 + 8 \cos^2 v}\Big(du^2 +\frac{dv^2}{1 + 8 \cos^2 v} \Big),
\end{align*}
$0 \leq u < \pi /2, 0 < v \leq \pi$ and has the following value:
\begin{align*}
\Lambda_1(\mathbb{KL})=12\pi E(2\sqrt{2}/3)\approx 13.365\pi,
\end{align*}
where $E(\cdot)$ is a complete elliptic integral of the second kind.
\end{corollary}

There are inconsistencies in the literature regarding whether there is a complete proof that the extremal metric on the Klein bottle found in \cite{MR2209284} is indeed maximal. See, for instance, Remark 1.1 in \cite{MR2259925}. One of the goals of the present article is to eliminate this inconsistency.

\medskip

\noindent {\bf Paper outline.} The rest of the paper is organized as follows: In Section \ref{background} we provide the necessary background for studying branched minimal immersions into spheres and recall the definition of conformal volume. Section \ref{Proof1} contains the proofs of Theorem \ref{unique} and Proposition~\ref{Xi}.
In Section \ref{Klein_bottle} we prove Theorem \ref{TSmooth} and that the conformal spectrum is 
continuous on the moduli space of conformal classes of metrics on $\Sigma$. 
\medskip

\noindent {\bf Acknowledgements.} The authors are grateful to I. Polterovich and A. Girouard for suggesting this problem. The authors are also grateful to I. Polterovich, A. Girouard, G. Ponsinet, A. Penskoi, G. Kokarev, and Alex Wright for fruitful discussions and especially to A. Girouard for a careful first reading of this manuscript. 
This research is part of the third author's PhD thesis at the Universit\'{e} de Montr\'{e}al under the supervision of Iosif Polterovich.

\section{Background}
\label{background} 

\subsection{Branched immersions and conical singularities}\label{branched_immersions}
Given a surface $\Sigma$ endowed with a conformal structure one defines a metric $g$ with conical singularities by declaring that at finitely many points $\{p_1,...,p_N\} \subset \Sigma$ (which are referred to as {\it conical points}) the metric has the following form in conformal coordinates centered at $p_i$: $\rho_i(z)|z|^{2\beta_i} |dz|^2$, where $\beta_i >-1$ and $\rho_i(z)>0$ is smooth. The metric is singular in the sense that it becomes degenerate at the conical points. This approach is taken, for instance, in \cite{troyanov}. One can check that if $\rho = 1$ near the conical points then $g$ is isometric to a cone with cone angle $2 \pi (\beta_i+1)$. In this article we are primarily concerned with metrics with conical singularities that arise from branched minimal immersions into spheres. A good introductory reference for branched minimal immersions is \cite{royden}.

  Fix a compact surface $\Sigma$ equipped with a smooth Riemannian metric $g_0$ (without conical points). Let $\Phi \colon (\Sigma,g_0) \to (\mathbb{S}^n, g_{\mathrm{can}})$ be a smooth map that is harmonic and conformal away from points at which $D_p \Phi = 0$. We will refer to points $p$ at which $D_p \Phi = 0$ as {\it branch points} and call $\Phi$ a {\it branched conformal immersion}. Note that away from the branch points the quadratic form $g = \Phi^*g_{\mathrm{can}}$ is actually an inner product on the tangent space and makes $\Phi$ a minimal immersion. Thus, we say that $\Phi \colon (\Sigma,g) \to (\mathbb{S}^n, g_{\mathrm{can}})$ is a {\it branched minimal immersion} into a sphere. We will see that $g$ possesses conical singularities at the branch points. 

In a neighborhood of a branch point we can choose conformal coordinates $z = z_1+iz_2$ on $\Sigma$ centered at $p$ and coordinates $x_1,...,x_n$ centered at $\Phi(p)$ such that $\Phi(z)$ takes the form:
\begin{align*}
x_1+ix_2 = z^{m+1}+\sigma(z) \\
x_k = \phi_k(z); \hspace{5 pt} k \geq 3,
\end{align*}
for $m \geq 1$
such that $\sigma(z)$ and $\phi_k(z)$ are $o(|z|^{m+1})$ and $\frac{\partial \sigma}{\partial z_j}$ and $\frac{\partial \phi_k}{\partial z_j}$ are $o(|z|^{m})$ as $z \to 0$ (that this is possible follows from the discussion found in \cite[Section 2]{royden}). The integer $m$ is referred to as the {\it order} of the branch point. Moreover, there exist $C^{1,\alpha}$-coordinates (see Lemma 1.3 of \cite{royden}) $\widetilde{z}$, which we will refer to as {\it distinguished parameters}, in which the map $\Phi$ takes the form:
\begin{align*}
x_1+i x_2 = \widetilde{z}^{m+1}\\
x_k = \psi_k(z); \hspace{5 pt} k \geq 3,
\end{align*}  
with $\psi_k(z)$ possessing the same asymptotics as $\phi_k(z)$ as $z \to 0$. Note that the distinguished parameters are not an admissible coordinate system for the smooth structure on $\Sigma$ since $\widetilde{z}$ is related to $z$ by $\widetilde{z} = z \left[ 1+ z^{-(m+1)}\sigma(z) \right]^{1/(m+1)}$. By looking at the form $\Phi$ takes in distinguished parameters it is clear that $D_p \Phi \neq 0$ in a punctured neighborhood of a branch point. Thus, branch points are isolated. Moreover, since regular points form an open set, $\Sigma$ can only posses finitely many branch points. From the previous discussion we see that in conformal coordinates centered at a branch point the metric is of the form $\rho(z)|z|^{2m} |dz|^2$, with $\rho(z)>0$ smooth. In other words, near the branched point  the metric is conformal to the Euclidean cone of total angle $2\pi (m +1)$. We will refer to $m$ as the {\it order} of the conical singularity and will also refer to the branch points $p$ as {\it conical points}.  

We recall the following lemma, which allows one to define the tangent space to $\Phi(\Sigma)$ at the image of a conical point $\Phi(p)$. For simplicity, we state the lemma in the setting of branched conformal immersions into a round sphere. However, it holds in greater generality (see Lemma 3.1, 3.2 and the remark on page 771 of \cite{royden} for the proof).
\begin{lemma}
\label{lemma:continuity}
Let $\Phi \colon \Sigma \to \mathbb{S}^n$ be a branched minimal immersion into a round sphere with a branch point at $p \in \Sigma$. Let $w$ and $x$ be distinguished parameters at $p$ and $\Phi(p)$, respectively. Define the tangent space to $\Phi(p)$ in distinguished parameters as the $x_1,x_2$-plane.
\begin{enumerate}[(i)]
\item If $\{p_n\}$ is a sequence in $\Sigma$ such that $p_n \to p$ and $\Phi$ is regular at $p_n$, then the tangent plane to $\Phi(\Sigma)$ at $\Phi(p_n)$ tends to the $x_1,x_2$-plane in distinguished parameters. Consequently, the Gauss map, which assigns to each point $q \in \Sigma$ the tangent plane to $\Phi(\Sigma)$ at $\Phi(q) \in \mathbb{S}^n$ is continuous on all of $\Sigma$. 
\item The definition of the tangent space to $\Phi(p)$ does not depend on the choice of distinguished parameters. 
\end{enumerate} 
\end{lemma}

Let $g$ be a metric on $\Sigma$ with conical singularities. Thus, $g = f g_0$, where $g_0$ is a smooth Riemannian metric on $\Sigma$ and $f$ is a smooth function on $\Sigma$ that is nonzero except at possibly finitely many points. One can define the first Laplace eigenvalue corresponding to this singular metric using the variational characterization:
\begin{equation} \label{eigenvalue}
\lambda_1(g) = \inf_{\stackrel{u\in H^1(\Sigma, g)}{ u\perp 1}} R(u, g),
\end{equation}
where:
\begin{equation*}
R(u, g) = \frac{ \int_{\Sigma} | \nabla u |^2 \dvol(g_0) }{  \int_\Sigma u^2 \dvol(g)}
\end{equation*}
is the Rayleigh quotient and $H^1(\Sigma, g)$ is the completion of the set:
\begin{align*}
\left \{ u \in L^2(\Sigma, \dvol(g)); \quad \int_{\Sigma} |\nabla u |^2 \dvol(g_0) <\infty \right \}
\end{align*}
with respect to the norm:
\begin{align*}
 \| u \|^2_{H^1(g)} =  \int_\Sigma u^2 \dvol(g) +  \int_{\Sigma} | \nabla u |^2 \dvol(g_0). 
\end{align*}
When a metric $g_0$ is smooth, we will regard $H^1(\Sigma, g_0)$ as the usual Sobolev space. Note that, essentially by the conformal invariance of the Dirichlet energy,  
 $H^1(\Sigma,g) = H^1(\Sigma,g_0)$, meaning that they are equal as sets, and the norms define the same topology  (see \cite[Proposition 3]{troyanov}).
A function $u \in H^1(\Sigma, g)$ for which the infimum of the Rayleigh quotient in (\ref{eigenvalue}) is achieved is called \textit{a first eigenfunction}. For metrics with conical singularities, first eigenfunctions exist (see \cite{kokarev} Proposition 3.1). By the usual elliptic regularity argument (see, for instance, \cite[Corollary 8.11]{MR1814364}), the first eigenfunctions are smooth and satisfy the following equation:
\begin{align*}
\Delta_{g_0} u = \lambda_1(g) f u. 
\end{align*}

\subsection{Conformal volume}

The notion of conformal volume was introduced by Li and Yau to prove bounds on $\lambda_1$ that depend only on the genus \cite{MR674407}. It will be used in our poof of Theorem \ref{unique}. Throughout, let $G_n$ denote the group of conformal diffeomorphisms of the $n$-sphere with its canonical metric and let $\Phi \colon \Sigma \to \mathbb{S}^n$ be a conformal immersion with possible branch points.\begin{definition}
\begin{enumerate}[(i)]
\item The {\it conformal $n$-volume} of $\Phi$ is given by:
\begin{equation*}
\vol_c(n, \Phi) := \sup_{\gamma \in G_n} \vol(\Sigma, (\gamma \circ \Phi)^*g_{\mathrm{can}}).
\end{equation*} 
\item The {\it conformal $n$-volume} of $\Sigma$, denoted $\vol_c(n, \Sigma)$, is the infimum of $\vol_c(n, \Phi)$ over all branched conformal immersions $\Phi\colon \Sigma \to \mathbb{S}^n$. 
\end{enumerate}
\end{definition}
\begin{remark}
In the recent preprint~\cite{kokarev_volc} Kokarev used conformal volume to obtain bounds for higher eigenvalues $\lambda_k$.
\end{remark}

\section {Proofs of main results}\label{Proof1}

To prove Theorem \ref{unique} we follow the same steps used by El Soufi and Ilias to prove the analog of Theorem \ref{unique} for minimal immersions without branch points (see Corollary 3.3 in \cite{ilias}). While some propositions easily generalize to the setting of branched minimal immersions (compare Proposition \ref{n-conformal-volume} with \cite[Theorem 2.2]{ilias}) others do not generalize completely (compare Proposition \ref{auxiliary_prop} with \cite[Proposition 3.1]{ilias}).  

\begin{proposition}
\label{variation_formula}
Let $(\Sigma, g)$ be a compact Riemannian surface possibly with conical singularities and suppose that there exists a branched minimal isometric immersion $\Phi$ of $(\Sigma, g)$ into a round sphere of dimension $n$, then:
\begin{equation*}
\vol(\Sigma, g) = \vol_c(n, \Phi) \geq \vol_c(n, \Sigma).
\end{equation*}
Moreover, if $\Phi(\Sigma)$ is not an equatorial 2-sphere, then $\vol(\Sigma, g)> \vol(\Sigma, (\gamma \circ \Phi)^* g_{\mathrm{can}})$ for every $\gamma \in G \backslash \myO(n+1)$. 
\end{proposition}

\begin{proof}
Given a unit-vector $a \in \mathbb{R}^{n+1}$, let $A$ denote the projection of the vector onto the tangent space of each point $\mathbb{S}^n$. Then $A$ is the gradient vector field of the function $u = \langle \cdot, a \rangle$. Let $(\gamma_t^a)_t$ be the time-$t$ map for the flow associated to $A$. Then $(\gamma^a_t)^*g_{\mathrm{can}} = e^{2f} g_{\mathrm{can}}$, with $f$ a smooth function on the sphere. Recall that for every $\gamma \in G_n$ there exists $r \in \myO(n+1)$ and $\gamma_t^a$ such that $\gamma = r \circ \gamma_t^a$ (see the Lemma on page 259 of \cite{ilias}). Thus, it suffices to show that:
\begin{equation*}
\vol(\Sigma, (\gamma_t^a \circ \Phi)^*g_{\mathrm{can}}) \leq \vol(\Sigma, \Phi^*g_{\mathrm{can}}) \hspace{5 pt} \text{for every }a \in \mathbb{S}^n \text{ and }t\geq0. 
\end{equation*}

First, we need to verify that we can make use of the first variation formula for the area of $\Sigma$.  
Let $\{ p_1,..., p_N \}$ denote the branch points of $\Phi$. Then 
\begin{equation*}
\widehat{\Sigma} :=  \Sigma \backslash \{p_1,...,p_N \}
 \end{equation*}
 is an (open) Riemannian manifold and $\Phi$ induces a minimal isometric immersion of $\widehat{\Sigma}$ into $\mathbb{S}^n$. For convenience, we will often identify $\widehat{\Sigma}$ and its image under $\Phi$.
 In coordinates centered at a branch point of order $m$, the volume form is given by:
 \begin{align*}
\dvol((\gamma_t^a \circ \Phi)^* g_{\mathrm{can}})(z) = \rho(t,z) |z|^{2 m} |d z \wedge d \overline{z}|,
 \end{align*}
 where $\rho(t,z)$ is a smooth positive function. Thus, it is clear that the volume form is differentiable (smooth) in $t$ and the derivative with respect to $t$ is identically zero at the branch points. Set $\gamma = \gamma^a_{t_0}$. Away from the singular points, we have the usual expression for the derivative of the volume form:
 \begin{align}
 \label{variation1}
\restr{\frac{d}{dt}}{t=t_0} \dvol((\gamma^a_t \circ \Phi)^* g_{\mathrm{can}})(x) = - \left \langle A_{\gamma(x)}, H_{\gamma(x)}^{\gamma\left(\widehat{\Sigma} \right)} \right \rangle \dvol((\gamma \circ \Phi)^*g_{\mathrm{can}})(x)  \\
+ \mydiv_{\gamma (\widehat{\Sigma})} (A_{\gamma(x)}^\top)  \dvol((\gamma \circ \Phi)^*g_{\mathrm{can}})(x) \nonumber , 
 \end{align}
 where $H_{\gamma(x)}^{\gamma \left(\widehat{\Sigma} \right)}$ is the mean curvature vector for $\gamma \left(\widehat{\Sigma}\right)$ and $A^\top$ is the projection of $A$ onto the tangent space of $\gamma \left( \widehat{\Sigma} \right)$. Since $\Phi$ is minimal away from the branch points, one can compute the expression for the mean curvature vector explicitly away from the branch points (see page 260 of \cite{ilias}):
 \begin{align*}
 H_{\gamma(x)}^{\gamma \left(\widehat{\Sigma} \right)} = -2 e^{-2f} D\gamma((\nabla f^{\perp})_x).
 \end{align*}
 Moreover, by Lemma \ref{lemma:continuity}, $\nabla f^{\perp}$ extends to a continuous vector field on all of $\Phi(\Sigma)$. Thus, $H^{\gamma(\Sigma)}$ extends to a continuous vector field on the branch points. It follows that the first term in the right hand side of \eqref{variation1} is zero at a branch point. Thus, the second term in the right hand side extends to something continuous and zero at the branch points. 
 
 Now we integrate both sides of \eqref{variation1} to recover the usual first variation formula. However, since $A_x^{\top}$ only extends to a continuous vector field on $\Phi(\Sigma)$, some care is required to show that the integral of the second term in the right hand side of \eqref{variation1} is zero. Notice that the integrals of the left hand side and the first term in the right hand side of \eqref{variation1} converge as improper integrals. Thus, it suffices to exhibit an exhaustion of $\gamma(\widehat{\Sigma})$ by compact sets with smooth boundary such that the integral of the second term on the right hand side of \eqref{variation1} converges to zero. Let $\{\Omega_n\}_{n=1}^\infty$ be an exhaustion of $\widehat{\Sigma}$ by compact sets with smooth boundary such that for each connected component of $\partial \Omega_n$ there exist distinguished parameters $(z_1,z_2)$ centered at a singular point such that the image of the connected component of $\partial \Omega_n$ is given by $z_1^2+z_2^2 = \frac{1}{n^2}$, for $n$ large enough. Then $\left \{\gamma(\Omega_n) \right \}$ is an exhaustion of $\gamma(\widehat{\Sigma})$. The Divergence Theorem yields:
 \begin{align*}
 \int_{\gamma(\Omega_n)} \mydiv_{\gamma(\Omega_n)}\left(A^\top_x\right) \dvol(g_{\mathrm{can}}) &= \int_{\partial \gamma(\Omega_n)} \langle A^\top_x, N_x \rangle \mathrm{d}s \\
 &= \int_{\partial \Omega_n} \langle A^\top_{\gamma(x)}, N_{\gamma(x)} \rangle e^{2f} \mathrm{d}s \\
 &= \int_{\partial \Omega_n} \langle D \gamma (A_x)^\top, N_{\gamma(x)} \rangle e^{2 f} \mathrm{d}s \\
 &= \int_{\partial \Omega_n} \langle D \gamma (A_x^\top), N_{\gamma(x)} \rangle e^{2 f} \mathrm{d}s,
 \end{align*}
 where $N$ is the outward pointing unit normal vector field along $\partial \gamma(\Omega_n)$ and $\mathrm{d}s$ is the length element along $\partial \gamma(\Omega_n)$.
 Again by Lemma \ref{lemma:continuity}, $A_x^\top$ extends to a continuous vector field on $\Phi(\Sigma)$, then the Cauchy-Schwartz inequality shows that the integral is $O((1/n)^{2 m+1})$. Thus, we see that, as an improper integral, $\int_{\Sigma}   \mydiv_{\gamma(\Sigma)} (A^\top_{\gamma(x)})  \dvol((\gamma \circ \Phi)^*g_{\mathrm{can}})=0$.
 Integrating both sides of \eqref{variation1} yields the usual first variation formula:
 \begin{equation*}
\restr{\frac{d}{dt}}{t=t_0} \vol(\Sigma, (\gamma^a_t \circ \Phi)^*g_{\mathrm{can}}) = - \int_{\Sigma} \langle H_{\gamma(x)}^{\gamma(\widehat{\Sigma})}, A_{\gamma(x)} \rangle \enskip \dvol\left((\gamma \circ \Phi)^*g_{\mathrm{can}} \right).
\end{equation*}   
At this point, the calculation done in the proof of Theorem 1.1 of \cite{ilias} applies:
\begin{equation}
\label{variation2}
\restr{\frac{d}{dt}}{t=t_0} \vol(\Sigma, (\gamma_t^a \circ \Phi)^*g_{\mathrm{can}}) = 2 \int_{\Sigma} \frac{u - u \circ \gamma}{1 - u^2} |A^{\perp}|^2 \enskip \dvol((\gamma \circ \Phi)^*g_{\mathrm{can}}),
\end{equation}
where the integrand is extended by continuity at the branch points. Since $u(x) \leq u(\gamma(x))$, we conclude that:
\begin{equation*}
\restr{\frac{d}{dt}}{t=t_0} \vol(\gamma_t^a(\widehat{\Sigma})) \leq 0.
\end{equation*} 
Thus, $\vol(\gamma_t^a(\Phi(\Sigma)))$ is non-increasing. 

Now suppose that there exists $a$ and $t_0>0$ such that $\vol(\gamma_{t_0}^a(\Phi(\Sigma))) = \vol(\Phi(\Sigma))$. Then $\restr{\frac{d}{dt}}{t=s} \vol(\gamma_t^a(\Phi(\Sigma))) = 0$ for $0 \leq s \leq t_0$. From this observation and \eqref{variation2} we conclude that $A^{\perp} = 0$ on $\Phi(\Sigma)$. Thus, $A$ restricts to a vector field on $\widehat{\Sigma}$. Observe that the integral curves of $A$ are great circles inside $\mathbb{S}^n$ connecting $a$ and $-a$. Therefore, $a,-a \in \Phi(\Sigma)$. If $a$ is a regular value of $\Phi$, then $\Phi(\Sigma)$ is just given by the image of $T_a\Phi(\Sigma)$ under the Riemannian exponential map of $\mathbb{S}^n$ based at $a$. So $\Phi(\Sigma)$ is an equatorial 2-sphere. 

Now suppose that $a$ corresponds to a singular value of $\Phi$. Again, by Lemma \ref{lemma:continuity} we may define the tangent space to $\Phi(\Sigma)$ at $a$ in $T_a\mathbb{S}^n$. Let $V$ denote this subspace. Given $p \in \Phi(\widehat{\Sigma})$ sufficiently close to $a$, let $\alpha \colon [0,1] \to \mathbb{S}^n$ be the minimizing geodesic connecting $p$ and $a$. Then $\alpha((0,1))$ is contained in $\Phi(\widehat{\Sigma})$. Moreover, since $\alpha'(t)$ is in the tangent space to $\Phi(\widehat{\Sigma})$ for every $t\in(0,1)$, then $\alpha'(1) \in V$. This shows that $\Phi(\Sigma)$ is again the image of $V$ under the Riemannian exponential map of $\mathbb{S}^n$ at $a$. Thus, $\Phi(\Sigma)$ is a 2-sphere.         
\end{proof}

The following proposition is a generalization of Theorem 2.2 in \cite{ilias} to the setting of branched minimal immersions. See also Theorem 1 in \cite{MR674407}. The proof of Theorem 2.2 in \cite{ilias} carries through without changes to the setting of branched minimal immersions.

\begin{proposition}
\label{n-conformal-volume}
Suppose $(\Sigma, g)$ is a Riemannian surface with possible conical singularities. For all $n$ such that the conformal $n$-volume is defined, we have:
\begin{equation*}
\bar\lambda_1(\Sigma, g) \leq 2 \vol_c(n, \Sigma). 
\end{equation*}
Equality holds if and only if $(\Sigma, g)$ admits, up to homothety, a branched minimal immersion into a sphere by first eigenfunctions. 
\end{proposition}

We generalize Proposition 3.1 of \cite{ilias} to the setting of branched minimal immersions. However, the statement is complicated by the fact that the image of a branched minimal immersion can be an equatorial $2$-sphere.  

\begin{proposition}
\label{auxiliary_prop}
Let $(\Sigma, g)$ be a surface with possible conical singularities. Moreover, suppose that the metric $g$ is induced from a branched minimal immersion $\Phi$ into a sphere. Then every metric with possible conical singularities $\widetilde{g}$ conformal to $g$ satisfies the following:
\begin{align*}
\bar\lambda_1(\Sigma,\widetilde{g})\leq 2 \vol(\Sigma, g).
\end{align*}
Criteria for equality are as follows:
\begin{itemize}
\item If the image of $\Phi$ {\em is not} an equatorial $2$-sphere, then equality holds if and only if the components of $\Phi$ are first eigenfunctions and $\widetilde{g}$ is homothetic to $g$.

%Assume that the image of $\Phi$ \underline{\bf is not} an equatorial $2$-sphere. Then the equality holds if and only if $\widetilde{g}$ is homothetic to $g$ and the immersion is given by first eigenfunctions.
\item If the image of $\Phi$ {\em is} an equatorial $2$-sphere, then equality holds if and only if there exists a conformal automorphism $\gamma$ of $\mathbb{S}^2$ such that $\widetilde{g}$ is homothetic to $(\gamma\circ\Phi)^*g_{\mathrm{can}}$ and the components of $\gamma\circ\Phi$ are first eigenfunctions.

\end{itemize} 
%Moreover, the equality holds if and only if $\widetilde{g}$ is homothetic to $g$ and the immersion is given by first eigenfunctions.
\end{proposition}

\begin{remark}
Note that in the second case the coordinates of $\Phi$ are not necessarily first eigenfunctions, see Example~\ref{Bolza_example} below.
\end{remark}

\begin{proof}
%By hypothesis, $g$ comes from a (branched) minimal immersion into some round sphere. Denote this branched immersion by $\Phi$. 
Let $\widetilde{g}\in [g]$ be another metric with possible conical singularities from the conformal class of the metric $g$. Let $\Phi_i$ denote the $i$-th component function corresponding to $\Phi$ (as a map from $\Sigma$ to $\mathbb{S}^n \subset \mathbb{R}^{n+1}$). According to the proof of Theorem 1 in \cite{MR674407}, there exists a conformal automorphism of $\mathbb{S}^n$, denoted $\gamma$, such that for every $i$ we have:
\begin{align}
\label{ineq_0}
\int_\Sigma (\gamma \circ \Phi)_i \dvol(\widetilde{g}) = 0.
\end{align}
Set $\Psi = \gamma \circ  \Phi$. Then using the conformal invariance of the Dirichlet energy and the variational characterization of Laplace eigenvalues, we have:
\begin{align}
\label{ineq_1}
\lambda_1(\widetilde{g}) &\leq \frac{\sum_i \int_\Sigma |d \Psi_i |_{\widetilde g}^2\, \dvol{(\widetilde g)}}{\sum_i \int_{\Sigma} \Psi_i^2\, \dvol{(\widetilde{g})}} \\
&= \frac{2 \vol(\Sigma, (\gamma \circ \Phi)^* g_{\mathrm{can}})}{\vol(\Sigma, \widetilde{g})}\\
&\leq 2 \vol(\Sigma, g) \vol(\Sigma, \widetilde{g})^{-1}
\label{ineq_2},
\end{align}
where the second inequality follows from Proposition \ref{variation_formula}. The desired inequality follows. 

Now assume that the equality is achieved, i.e. inequalities~\eqref{ineq_1} and~\eqref{ineq_2} turn into equalities and one has 
\begin{align}
\label{ineq_3}
\bar\lambda_1(\Sigma,\widetilde{g})= 2 \vol(\Sigma, g).
\end{align} 
By Proposition~\ref{n-conformal-volume} $\bar\lambda_1(\Sigma, \widetilde{g})\leq 2 \vol_c(n, \Sigma)$. Combining with equality~\eqref{ineq_3}, this yields
\begin{align*}
\vol(\Sigma, g)\leq \vol_c(n,\Sigma).
\end{align*} 
At the same time, by Proposition~\ref{variation_formula}, the reverse inequality is true. Therefore, one has an equality.

Assume that the image is not an equatorial $2$-sphere. Then by Proposition~\ref{variation_formula} $\gamma$ is an isometry. Since any isometry of the sphere is linear, equality~\eqref{ineq_0} is satisfied with $\gamma = \mathrm{id}$. Together with equality~\eqref{ineq_1} this yields that coordinates of $\Phi$ are first eigenfunctions for the metric $\widetilde g$. If $\widetilde g = e^{2\omega}g$ then
\begin{align}
\label{ineq_4}
\lambda_1(\widetilde g) \Phi = \Delta_{\widetilde g}\Phi = e^{-2\omega}\Delta_g\Phi = e^{-2\omega}\lambda_1(g)\Phi.
\end{align}
We conclude that in this case $\omega$ is constant and $\widetilde g$ is homothetic to $g$.

Now suppose that the image is an equatorial $2$-sphere. In this case we cannot conclude that $\gamma$ is an isometry. Nevertheless, equalities in~\eqref{ineq_0} and~\eqref{ineq_1} imply that coordinates of $\Psi$ are first eigenfunctions for $\widetilde g$. Setting $g' = \Psi^*g_{\mathrm{can}}$ and $\widetilde g = e^{2\omega}g'$ we obtain similarly to~\eqref{ineq_4},
\begin{align*}
\lambda_1(\widetilde g) \Psi = \Delta_{\widetilde g}\Psi = e^{-2\omega}\Delta_{g'}\Psi = 2e^{-2\omega}\Psi.
\end{align*}
We conclude that in this case $\omega$ is constant and $\widetilde g$ is homothetic to $g'$.

Let us prove the converse to the equality statements. If the components of $\Phi$ are first eigenfunctions then $\lambda_1(g) = 2$. Since $\tilde g$ is homothetic to $g$ one has
$$
\bar\lambda_1(\Sigma,\widetilde g) = \bar\lambda_1(\Sigma, g) = 2\vol(\Sigma,g).
$$

Suppose that the image of $\Phi$ is an equatorial $2$-sphere. Set $\Psi = \gamma\circ\Phi$ then after rescaling we may assume $\widetilde g = \Psi^*g_{\mathrm{can}}$ and $\vol(\Sigma,g) = 4\pi|\deg\Phi| = 4\pi|\deg\Psi| = \vol(\Sigma,\widetilde g)$, since conformal transformations preserve the absolute value of the degree. If the components of $\Psi$ are first eigenfunctions then $\lambda_1(\tilde g) = 2$ and one has
$$
\bar\lambda_1(\Sigma,\widetilde g) =8\pi|\deg\Psi| = 8\pi|\deg\Phi| = 2\vol(\Sigma,g).
$$

%Now suppose we have the following equality: 
%\begin{align*}
%\lambda_1(\widetilde{g}) \vol(\Sigma, \widetilde{g}) = 2 \vol(\Sigma, g) = 2 \vol_c(n, \Sigma).
%\end{align*} 
%After scaling, we may assume that $\lambda_1(\widetilde{g}) = 2$ and $\vol(\Sigma, \widetilde{g}) = \vol(\Sigma, g)$. Thus, the inequalities in \eqref{ineq_1} and \eqref{ineq_2} are actually equalities. 
%As in the proof of Proposition \ref{n-conformal-volume}, we conclude that $\widetilde{g} = \Psi^*{\can}$. If the image of $\Phi$ is not $\Xi$, then by Proposition \ref{variation_formula} we conclude that $\gamma$ is an isometry and the proposition is proved. 
\end{proof}

The following proposition is proved in \cite[Theorem 6]{MR1178529}. We reprove it here using a slightly different approach.
\begin{proposition}
\label{auxiliary_prop2}
Suppose that $\Phi\colon\Sigma\to\mathbb{S}^2$ is a branched minimal immersion by first eigenfunctions. Then
\begin{itemize}
\item[(i)]
For any other conformal map $\Psi\colon\Sigma\to\mathbb{S}^2$ one has $|\deg\Psi| \geq |\deg\Phi|$;
\item[(ii)] 
 If $|\deg\Psi| = |\deg\Phi|$ then there exists a conformal transformation $\gamma$ such that $\Psi = \gamma\circ\Phi$.
\end{itemize}
\end{proposition}
\begin{proof}
This proposition is a direct corollary of Proposition~\ref{auxiliary_prop}. To prove (i), we apply Proposition \ref{auxiliary_prop} for metrics 
$\widetilde g = \Phi^*g_{\mathrm{can}}$ and $g = \Psi^*g_{\mathrm{can}}$. Then $\lambda_1(\widetilde g) = 2$ and we conclude
 $$
 8\pi |\deg\Phi| = \bar\lambda_1(\Sigma,\widetilde g)\leq 2\vol(\Sigma,g) = 8\pi |\deg\Psi|.
 $$
 
 If $|\deg\Phi| = |\deg\Psi|$, then we switch the roles of $g$ and $\widetilde g$ and observe that we have an equality, i.e. by Proposition~\ref{auxiliary_prop} there exists a conformal automorphism $\gamma_0$ such that 
 \begin{equation}
 \label{metric_equality}
 (\gamma_0\circ\Phi)^*g_{\mathrm{can}}= \Psi^*g_{\mathrm{can}}.
 \end{equation}
 
 We would like to show that it implies the existence of an isometry $I$ of $\mathbb{S}^2$ such that 
$\Psi = I\circ\gamma_0\circ\Phi$. 
 
\begin{lemma} 
Let $f$ and $h$ be two holomorphic maps $\Sigma\to\mathbb{S}^2$ (i.e. meromorphic functions) such that for any choice of local coordinates one has $|f_z| = |h_z|$. Then there exists $\alpha\in\mathbb{R}$ and $c\in\mathbb{C}$ such that $f = e^{i\alpha}h + c$.
\end{lemma}
\begin{proof}
First, note that the condition of the lemma implies that that $f$ and $h$ have the same singular sets.
Let $p\in\Sigma$ be any regular point of $f$ and $h$, i.e. $df(p)\ne 0$ and $dh(p)\ne 0$. Let $z$ be local coordinates, then there exists a real-valued function $\alpha(z)$ such that 
$ f_z = e^{i\alpha(z)}h_z$. Taking $\partial_{\bar z}$ of both parts we obtain
$$
i(\partial_{\bar z}\alpha)e^{i\alpha}h_z = 0.
$$
Since $h_z\ne 0$ in a neighborhood of $p$, one concludes that $\alpha$ is a real-valued holomorphic function. Thus, $\alpha$ is a constant. Integrating the equality $f_z = e^{i\alpha}h_z$, we obtain an equality $f = e^{i\alpha}h + c$ valid in a neighborhood of $p$. Since it is an equality between two meromorphic functions, by unique continuation it is valid everywhere on $\Sigma$.
\end{proof}

By taking conjugates if necessary, we can assume that $\gamma_0\circ\Phi$ and $\Psi$ are both holomorphic. Equality~\eqref{metric_equality} guarantees that we can apply the previous lemma to these functions. The conclusion of the lemma then provides a desired isometry $I$. Setting $\gamma = I\circ \gamma_0$ concludes the proof.
\end{proof}

\begin{example} 
\label{Bolza_example}
In this example we demonstrate that the application of conformal transformations does not in general preserve the property ``coordinate functions are the first eigenfunctions." 
%This fact could be proved using the results of Petrides~\cite{petrides} on the compactness of the set of metrics maximizing $\bar\lambda_1$ in a given conformal class.

Let $\mathcal S$ be a Bolza surface and let $\Pi\colon\mathcal S\to\mathbb{S}^2$ be the corresponding hyperelliptic projection.
By~\cite{nayatani}, $\Pi$ is given by first eigenfunctions. Let us consider instead $\Pi_t = \gamma_t\circ\Pi$, where $\gamma_t = \cfrac{z+it}{1-itz}$, $t\in[0,1)$ is a conformal transformation moving the points of $\mathbb{S}^2$ towards the point $i$ along the shortest geodesic (the point $-i$ does not move). We claim that for $t$ close to $1$ the first eigenvalue 
$\lambda_1(\mathcal S,\Pi_t^*g_{\mathrm{can}})$ is close to $0$. Informally, it can be explained in the following way. As $t$ tends to $1$ the images of the branch points of $\Pi_t$ are getting closer and closer together. As a result, for large $t$ the surface $(\Sigma,\Pi_t^*g_{\mathrm{can}})$ looks like two spheres glued together with three small cylinders. To make this argument precise, we prove the following proposition.
\end{example}

\begin{proposition}
Suppose that $\Phi\colon\Sigma\to \mathbb{S}^2$ is a holomorphic map of degree $d$ such that the images of all the branch points lie in an open disk $D_r$ of radius $r$. Then $\lambda_{d-1}(\Sigma,\Phi^*g_{\mathrm{can}}) = o(1)$ as $r\to 0$. 
\end{proposition}
\begin{proof}
Let $p$ be a center of $D_r$ and let $\pi$ be a stereographic projection onto $\mathbb{C}$ from $-p$. Then $\pi(D_r)$ is a Euclidean ball $B_\rho(0)$ of radius $\rho = 2\tan\frac{r}{2}$ with center at $0$. Note that $\rho = O(r)$ as $r\to 0$.
Moreover, the variational capacity of $B_\rho(0)$ in $B_1(0)$ is $2\pi|\ln \rho|^{-1}$. Therefore, there exists a function $f_r\in H^1_0(B_1(0))$ such that $f_r \equiv 1$ on $B_\rho(0)$ and the Dirichlet energy of $f_r$ is $o(1)$. Let $h_r = 1 - \pi^*f_r$. Then $h_r\equiv 0$ on $D_r$, $h_r\equiv 1$ on a hemisphere and by conformal invariance of the Dirichlet energy
$$
\int_{\mathbb{S}^2}|\nabla h_r|^2\, \dvol(g_{\mathrm{can}}) = o(r).
$$

Outside $\Phi^{-1}(D_r)$ the map $\Phi$ is a covering map. Since $\mathbb{S}^2\backslash D_r$ is a topological disk, all its covering spaces are trivial. Therefore $\Phi^{-1}(D_r)$ coincides with $d$ copies of $\mathbb{S}^2\backslash D_r$. Let $h_{1,r},\ldots,h_{d,r}$ be functions $h_r$ considered as functions on their own copy of $\mathbb{S}^2\backslash D_r$. We can extend them by zero and consider as functions on $\Sigma$. Then their support is disjoint and
$$
\frac{\int_{\Sigma}|\nabla h_{i,r}|^2\,\dvol(\Phi^*g_{\mathrm{can}})}{\int_{\Sigma}h_{i,r}^2\,\dvol(\Phi^*g_{\mathrm{can}})}= \frac{\int_{\mathbb{S}^2}|\nabla h_{r}|^2\,\dvol(g_{\mathrm{can}})}{\int_{\mathbb{S}^2}h_{r}^2\,\dvol(g_{\mathrm{can}})}\leq \frac{o(1)}{2\pi} = o(1).
$$
The standard argument with min-max characterization of the eigenvalues concludes the proof.
\end{proof}

We see that for $t$ close to $1$ all branch points of $\Pi_t$ will concentrate in a small disk around $i$. Since $\deg\Pi_t=2$ the previous proposition yields that $\lambda_1(\mathcal S, \Pi_t^*g_{\mathrm{can}}) \to 0$ as $t\to 1$. Note that this argument works with the point $i$ replaced by an arbitrary point distinct from the branch point of $\Pi$.

\begin{proposition}
\label{dichotomy} 
Let $\Psi\colon \Sigma\to \mathbb{S}^2$ be a conformal map. Suppose  $\Phi\colon \Sigma\to\mathbb{S}^n$ is such that $\Phi^*g_{\mathrm{can}} = \Psi^*g_{\mathrm{can}}$. Then the image of $\Phi$ lies in an equatorial $2$-sphere. 
\end{proposition}
\begin{proof}
Let $g = \Psi^*g_{\mathrm{can}} = \Phi^*g_{\mathrm{can}}$. Let $II$ denote the second fundamental form of $\Phi$. Then Gauss' equation implies that at any regular point
$$
1 = K + \frac{1}{2}|II|^2_g.
$$ 
At the regular point of $g$ one has $K = 1$, therefore, $II=0$, i.e. $\Phi(\Sigma)$ is totally geodesic in a neighborhood of any smooth point. Thus, that neighborhood gets mapped to a piece of an equatorial $2$-sphere. The conclusion follows from the following standard open-closed argument.

Fix a regular point $p\in\Sigma$. Since $\Phi$ is totally geodesic in a neighbourhood $U_p$ of $p$ there exists a $3$-dimensional subspace $E_p$ such that $D\Phi(TU_p)\subset E_p$. Let $\Sigma_{\mathrm{reg}}\subset \Sigma$ be the set of regular points. Define $V_p$ to be the set of $q\in\Sigma_{\mathrm{reg}}$ such that $D\Phi(T_q\Sigma)\subset E_p$. Then $V_p$ possesses the following properties.

{\bf Non-empty.} Indeed, $U_p\subset V_p$.

{\bf Open.} Indeed, let $q\in V_p$. On one hand, it means that $D\Phi(T_q\Sigma)\subset E_p$. On the other, it is always true that $D\Phi(T_q\Sigma)\subset E_q$. Since $D\Phi(T_q\Sigma)$ is $2$-dimensional, it follows that $E_p=E_q$. Therefore, $U_q\subset V_p$.

{\bf Closed.} Indeed, the complement to $V_p$ has the form $\cup V_q$ for some $q\in\Sigma_\mathrm{reg}$. Therefore, it is open.

Finally, we remark that $\Sigma_{\mathrm{reg}}$ is $\Sigma$ with finitely many points removed. Thus, it is connected. Therefore, $V_p = \Sigma_{\mathrm{reg}}$ and by continuity $\Phi(\Sigma)\subset E_p$. 
\end{proof}

\begin{proof}[Proof of Theorem \ref{unique}]
 Assume that $\Phi_1$ and $\Phi_2$ are branched minimal immersions by first eigenfunctions. Then from Proposition  \ref{n-conformal-volume}  we have: 
\begin{align*}
2\vol(\Sigma, g_1) = 2\vol(\Sigma,g_2), 
\end{align*}
where $g_1:= \Phi_1^* g_{\mathrm{can}}$ and $g_2:= \Phi_2^*g_{\mathrm{can}}$.
However, if neither of the images of $\Phi_1$ and $\Phi_2$ are equatorial $2$-spheres then by Proposition \ref{auxiliary_prop} this is only possible if the metrics $g_1$ and $g_2$ are homothetic. Since their volumes coincide, they are equal.
 
Suppose that the image of the map $\Phi_1$ lies in a $2$-sphere and the image of the map $\Phi_2$ does not. Let $g_1$ and $g_2$ be the corresponding induced metrics. First, note that $\lambda_1(g_1) = \lambda_1(g_2) = 2$. Second, by Proposition~\ref{n-conformal-volume} one has $\bar\lambda_1(\Sigma,g_1) = \bar\lambda_1(\Sigma,g_2)$. Then by Proposition~\ref{auxiliary_prop} applied to $\Phi_2$ we conclude that $g_1$ is homothetic to $g_2$. Moreover, they have the same first eigenvalue, therefore, $g_1 = g_2$. The conclusion follows from Proposition~\ref{dichotomy}.
\end{proof}
 
 The aim of the following proposition is to show that there is no conformal class which falls into category 3) of Theorem \ref{unique} when the surface is a $2$-torus.  
 
\begin{proposition}\label{Xi}
Let $\Sigma$ be a $2$-torus and $\Phi: \Sigma \to \mathbb{S}^n $ be a non-constant branched minimal immersion by first eigenfunctions. Then the image of $\Phi$ cannot be an equatorial $2$-sphere.
\end{proposition}

\begin{proof}
This proposition is stated as obvious in Montiel and Ros~\cite[Corollary 8(b)]{MR1178529}. However, we were unable to come up with an obvious explanation of this fact. Instead, we provide a proof below.
Suppose that there exists a branched minimal map $\Phi\colon \mathbb{T}^2\to \mathbb{S}^2$ by first eigenfunctions of the metric $g$. After possibly taking a conjugate, we may assume that $\Phi$ is holomorphic, i.e. is given by a meromorphic function $f$. 

{\bf Claim 1.} $\deg f = 2$.

By inequality~\eqref{Li-Yau}, $\deg f\leq 2$. At the same time, any meromorphic function of degree one is invertible which is impossible for $f$ since $\mathbb{T}^2\not\approx\mathbb{S}^2$.

{\bf Claim 2.} For any two meromorphic functions $f,h$ of degree $2$ there exists a holomorphic automorphism $\gamma$ of $\mathbb{S}^2$ such that $f = \gamma \circ h$.

This immediately follows from Proposition~\ref{auxiliary_prop2}.

{\bf Claim 3.} For any point $p\in\mathbb{T}^2$ there exists a meromorphic function $f_p$ of degree $2$ such that its only pole has degree $2$ and is located at $p$. 

Let $\Lambda$ be a full rank lattice in $\mathbb{C}$ and suppose that $g$ is conformal to the flat metric on the torus $\mathbb{C}/ \Lambda$. Then we may take $f_p$ to be $\wp(x-p)$, where $\wp$ is the Weierstrass elliptic function corresponding to $\Lambda$ (for a definition of the Weierstrass elliptic function, see \cite[Section 6.2]{MR2856237}). 

Let $p\ne q$ and let $\gamma(z) = \frac{az+b}{cz+d}$ be such that $f_p = \gamma\circ f_q$. Then 
\begin{equation}
\label{fpfq}
f_p(cf_q + d) = af_q + b
\end{equation}

{\bf Claim 4.} The divisor of $h = cf_q + d$ is $2p-2q$, i.e. the only zero of $h$ is $p$, its order is $2$; and the only pole of $h$ is $q$, its order is $2$.

The function $f_p$ has a pole of order $2$ at $p$ but the left hand side of~\eqref{fpfq} is finite at $p$. Therefore, $h$ has a zero of order at least $2$ at $p$. At the same time, $\deg h\leq 2$, so $p$ is the unique zero and is of order exactly $2$. Similarly, $h^{-1}$ has a unique zero of order $2$ at $q$. 

By Abel's Theorem (see \cite[Section 5.9]{MR1909701}), there exists $h$ such that $(h) = 2p-2q$ iff $2p-2q = 0$ as points in $\mathbb{C}/\Lambda$. We arrive at a contradiction since $p$ and $q$ were chosen arbitrary.
\end{proof}

\medskip
 
\section{Application to the 2-torus and the Klein bottle}\label{Klein_bottle} 

\subsection{Conformal degeneration on the $2$-torus and maximal metrics.} \label{Conf&Max torus} It is well-known that any metric on the 2-torus is conformally equivalent to a flat one obtained from the Euclidean metric on $\mathbb{C}$ under factorization by some lattice $\Gamma \subset \mathbb{C}$ generated by 1 and $a+ib \in \mathcal{M}$, where 
\begin{align*}
\mathcal{M}:=\{a+ib\in \mathbb{C} \vert 0 \leq a \leq 1/2, a^2+b^2 \geq 1, b>0\}.
\end{align*}
Thus, conformal classes are encoded by points of $\mathcal{M}$ (\textit{the moduli space of flat tori}).

We point out the following action of a subgroup of the group of conformal diffeomorphisms isomorphic to $\mathbb{S}^1$. Let $\mathbb{C}/\Gamma$ where $\Gamma$ is generated by 1 and $a+ib \in \mathcal{M}$.  For $\theta \in \mathbb{R}$ we have an action on $\mathbb{C}$ via translation: $s_{\theta}(x+iy) = x+\theta +i y$. This $\mathbb{R}$-action on $\mathbb{C}$ induces an $\mathbb{S}^1$-action on $\mathbb{C}/\Gamma$ that has no fixed points. A metric in the conformal class corresponding to $a+ib \in \mathcal{M}$ is given by $f(x + i y) (dx^2+dy^2)$ where $f(z)$ is a smooth positive function that is invariant under the action of $\Gamma$. Since $s_{\theta}$ is a translation we have: $s_\theta^*(f(x + i y) (dx^2+dy^2)) = f(x+\theta + i y)(dx^2+dy^2)$. Thus, $s_\theta$ acts by conformal diffeomorphisms.     
We recall the following result concerning maximization of $\lambda_1$ and conformal degeneration.

\begin{theorem}[\cite{MR2514484}] \label{Girouard torus}Let $(g_n)$ be a sequence of metrics of area one on the 2-torus such that the corresponding sequence $(a_n+ib_n) \in \mathcal{M}$ satisfies $\lim_{n \to \infty}b_n=\infty$, then

\begin{align*}
\lim_{n \to \infty}\lambda_1(g_n)\leq 8\pi.
\end{align*}

\end{theorem}

\subsection{Conformal degeneration on the Klein bottle and maximal metrics.}\label{Conf&Max}
We define the Klein bottle as the quotient of $\mathbb{C}$ under the action of the group $G_b$, generated by the following elements:
\begin{align*}
t_{b}(x+iy)=x+i(y+b); \hspace{15 pt} \tau (x+iy)=x+\pi-iy.
\end{align*}
As a consequence of the Uniformization Theorem, any metric on the Klein bottle is conformal to a flat metric on $K_b:= \mathbb{C}/G_b$ for some $b>0$. Thus, the moduli space of conformal classes of metrics on the Klein bottle is encoded by the positive real numbers. Similar to the case of the $2$-torus there is a group of conformal diffeomorphisms isomorphic to $\mathbb{S}^1$. Indeed, translations of the form $x+iy \mapsto x+\theta +iy$ induce an action of $\mathbb{R}/\pi \mathbb{Z}$ on $K_b$ without fixed points. Just as above this action induces an action by conformal diffeomorphisms. 
We recall the following result:

\begin{theorem}[\cite{MR2514484}]\label{Girouard} Let $(g_{n}) \subset \mathcal{R}(\mathbb{KL})$ be a sequence of metrics of area one on the Klein bottle.
\begin{enumerate}[(i)]
\item If $\lim_{n \to \infty} b_{n} =0$, then $\lim \sup_{n \to \infty} \lambda_{1}(g_{n}) \leq 8\pi$.
\item If $\lim_{n \to \infty} b_{n} =\infty, then \lim \sup_{n \to \infty} \lambda_{1}(g_{n}) \leq 12\pi$.
\end{enumerate}
\end{theorem}

Roughly speaking, Theorems \ref{Girouard torus} and \ref{Girouard} prove that the maximal metrics for the functional $\bar\lambda_1$ on the 2-torus and the Klein bottle must be in a conformal class which corresponds to a fundamental domain which cannot be too ``long and skinny."

\subsection{Continuity results}\label{continuity}

One of the classical distances considered on the moduli space of complex structures is the \emph{Teichm\"uller distance}. Naturally, this distance induces a distance $d_T$ on the space of conformal classes. In this section we show that the 
conformal eigenvalues 
$$
\Lambda_{k}(\Sigma,[g]):=\sup_{g'\in [g]}\lambda_{k}(g')\myVol(\Sigma, g')
$$
are continuous on the space of conformal classes.
This fact should be well-known but we were not able to find a reference.

Here we follow \cite{MR2850125}. First, we define the Teichm\"uller distance for orientable surfaces.
We define a notion of \emph{dilatation}. Let $f\colon\Sigma_1\to\Sigma_2$ be an orientation-preserving homeomorphism between two Riemann surfaces which is a diffeomorphism outside a finite set of points. The function $k_f(p)$ of $f$ at $p$ is defined in local coordinates as $k_f(p) = \frac{|f_z| + |f_{\bar z}|}{|f_z| - |f_{\bar z}|}$. It is defined only at points where $f$ is smooth and does not depend on the choice of coordinates. One defines the dilatation of $f$ by the formula $K_f = ||k_f||_\infty$. If $K_f<\infty$ then one says that $f$ is $K_f$-quasiconformal.
%Moreover, given metrics $g_i$ compatible with complex structure on $\Sigma_i$, one has 
%
%\begin{equation}
%\label{dilatation}
%\left(K_f(p)\right)^2 = \frac{\max\limits_{v\in T_p\Sigma_1,\,g_1(v,v)=1}(f^*g_2)(v,v)}{\min\limits_{v\in T_p\Sigma_1,\,g_1(v,v)=1}(f^*g_2)(v,v)}
%\end{equation}

For any holomorphic quadratic differential $q_1$ on a Riemann surface $\Sigma_1$ its absolute value $|q_1|$ defines a flat metric with conical singularities at zeroes of $q_1$ compatible with the complex structure. For any non-singular point $p_0\in\Sigma_1$ one can define \emph{natural coordinates} $\eta = \int_{p_0}^p\sqrt{q_1}$ such that $q_1 = d\eta^2$ and $|q_1| = |d\eta|^2$, i.e. the metric is Euclidean in these coordinates. Natural coordinates are defined up to a sign and a translation. 

A homeomorphism $f\colon\Sigma_1\to\Sigma_2$ between Riemann surfaces is called a \emph{Teichm\"uller mapping} if there exist holomorphic differentials $q_1$ on $\Sigma_1$ and $q_2$ on $\Sigma_2$ and a real number $K>1$ such that
\begin{itemize}
\item[(i)] $f$ maps zeroes of $q_1$ to zeroes of $q_2$;
\item[(ii)]  If $p$ is not a zero of $q_1$ then with respect to a set of natural coordinates for $q_1$ and $q_2$ centered at $p$ and $f(p)$ respectively, the mapping $f$ can be written as
$$
f(z) = \frac{1}{2}\left(\frac{K+1}{\sqrt{K}}z + \frac{K-1}{\sqrt{K}}\bar z\right),
$$
or, equivalently,
$$
f(x+iy) = \sqrt{K}x + i\frac{1}{\sqrt{K}}y
$$
\end{itemize}
In particular, a Teichm\"uller map has dilatation $K$ and is smooth outside of zeroes of $q_1$. Moreover, in natural coordinates $|q_1| = dx^2 + dy^2$ and $f^*|q_2| = Kdx^2 + \frac{1}{K}dy^2$. 

\begin{theorem} [Teichm\"uller's Theorem]
\label{teichmuller}
Given an orientation preserving homeomorphism $f\colon\Sigma_1\to\Sigma_2$ between non-isomorphic Riemann surfaces there exists a Teichm\"uller mapping homotopic to $f$. It is unique provided $\chi(\Sigma)<0$. If $\chi(\Sigma) = 0$ then a Teichm\"uller mapping is affine and is unique up to a translation, therefore the dilatation is independent of the choice of the mapping.
\end{theorem}

\begin{definition}
Let $\Sigma$ be an orientable surface. Consider two different complex structures on $\Sigma$ making it into Riemann surfaces $\Sigma_1$ and $\Sigma_2$. Then one defines the \emph{Teichm\"uller distance} between $\Sigma_1$ and $\Sigma_2$ as follows,
$$
d_T(\Sigma_1,\Sigma_2) = \frac{1}{2}\inf\limits_f\log(K_f),
$$
where $f$ ranges over all Teichm\"uller mappings $f\colon \Sigma_1\to\Sigma_2$.
%
%
%
%Let $[g_1]$ and $[g_2]$ be two conformal classes of metrics on $\Sigma$. Let $\Sigma_i$ denote $\Sigma$ considered as a Riemann surface with complex structure associated with $g_i$ inducing the same orientations on $\Sigma$. Then
%$$
%d_T([g_1],[g_2]) = \frac{1}{2}\inf\limits_f\log(||K_f||_{\infty}),
%$$
%where infimum is taken over all diffeomorphisms $f\colon\Sigma_1\to\Sigma_2$.
\end{definition}
\begin{remark}
The fact that $d_T$ is indeed a distance function is not obvious and relies on proper discontinuity of the action of the mapping class group on the Teichm\"uller space.
\end{remark}

Teichm\"uller distance $d_T$ on the moduli space of complex structures induces a distance function on the moduli space of conformal classes. Indeed, let $[g_1]$ and $[g_2]$ be two conformal classes. Choose complex structures $\Sigma_i$ compatible with $[g_i]$ inducing the same orientation on $\Sigma$. Then one sets $d_T([g_1],[g_2]) = d_T(\Sigma_1,\Sigma_2)$. 

Up until now we considered orientable surfaces $\Sigma$. Let us now address the case of non-orientable $\Sigma$. Denote by $\pi\colon\hat\Sigma\to\Sigma$ an orientable double cover and by $\sigma$ a corresponding involution exchanging the leaves of $\pi$. Let $[g_1]$ and $[g_2]$ be two conformal classes on $\Sigma$. Choose two complex structures $\hat\Sigma_1$ and $\hat\Sigma_2$ on $\hat\Sigma$ compatible with $[\pi^*g_1]$ and $[\pi^*g_2]$. Then one defines 
$$
d_T([g_1],[g_2]) = \frac{1}{2}\inf\limits_f\log(K_f),
$$ 
where $f$ ranges over all Teichm\"uller mappings $f\colon\hat\Sigma_1\to\hat\Sigma_2$ commuting with $\sigma$.
\begin{remark}
This definition is implicitly making use of the equivariant version of Teichm\"uller's Theorem. If in Theorem~\ref{teichmuller} $f\circ\sigma$ is homotopic to $\sigma\circ f$ then the Teichm\"uller mapping can be chosen $\sigma$-equivariant. Indeed, suppose that $h$ is the Teichm\"uller mapping for $f$. Then $h \circ \sigma$ is the Teichm\"uller map for $f \circ \sigma$. Similarly, $\sigma \circ h$ is the Teichm\"uller map for $\sigma \circ f$. If $\sigma \circ f$ is homotopic to $f \circ \sigma$, then $h \circ \sigma$ must be homotopic to $\sigma \circ h$, and by the uniqueness part of Teichm\"uller's Theorem we obtain $\sigma \circ h = h \circ \sigma$. 
\end{remark}

%The space of conformal structures on the surface $\Sigma$ is called \textit{the Teichm\"uller space}. It is a metric space with \textit{the Teichm\"uller distance} $d_T([g],[h])=\frac{1}{2}\log K$, where $K$ is the dilatation of the Tecim\"uller map. Recall that the map $\Phi$ is said to be the Teichm\"uller map if its Beltrami coefficient $\mu(z)=\frac{\partial \Phi}{\partial \bar{z}} / \frac{\partial \Phi}{\partial z}$ has a constant norm. Here, $z$ is a local complex coordinate on $\Sigma$. The dilatation $K$ is defined by the formula $K=\frac{1+||\mu||_\infty}{1-||\mu||_\infty}$. Teichm\"uller's theorem states that in the isotopy class of the map $\Phi$ there exists a unique Teichm\"uller map. We have two distances on the Teichm\"uller space. The Teichm\"uller spaces of flat tori and of flat Klein bottles that we considered above are endowed with natural distances. These distances are the Teichm\"uller distances. 

%However, Proposition \ref{continuity} is proved for the quasi isometric distance on the Teichm\"uller space. The following proposition asserts that these two distances coincide in the case of closed compact surfaces. 

\begin{proposition}\label{cont}
The conformal eigenvalues $\Lambda_k(\Sigma,[g])$ are continuous in the distance $d_T$.
\end{proposition}
\begin{proof}
We follow the notation of~\cite{kokarev}. Namely, given a conformal class $c$ of metrics and a measure $\mu$ on $\Sigma$ we define the Rayleigh quotient
$$
R_c(u,\mu) = \frac{\int_\Sigma|\nabla u|^2_g\,\mathrm{dV}(g)}{\int_\Sigma u^2\,d\mu}
$$
and the eigenvalues $\lambda_k(c,\mu)$ as critical values of the Rayleigh quotient. For a comprehensive study of eigenvalues in this context, including a proof of the existence of eigenfunctions, see~\cite{kokarev}.

We start with the case of an orientable $\Sigma$. Let $[g_1]$ and $[g_2]$ be two conformal classes and let $\Sigma_1$ and $\Sigma_2$ be the corresponding Riemann surfaces. Denote by $f\colon\Sigma_1\to\Sigma_2$ any Teichm\"uller map, let $q_1$ and $q_2$ be the corresponding quadratic differentials and suppose that $S_1$ and $S_2$ are their zeroes respectively. By property (ii) at any point of $\Sigma_1\backslash S_1$ one has 
\begin{equation}
\label{forms}
\frac{1}{K}f^*|q_2|\leq |q_1|\leq Kf^*|q_2|.
\end{equation}

At this point we use the conformal invariance of the Dirichlet energy. 
Let $K_i\subset K_{i+1}$ be a compact exhaustion of $\Sigma_2\backslash S_2$. Similarly, $f^{-1}(K_i)$ form a compact exhaustion of $\Sigma_1\backslash S_1$. Then for any $u\in H^1(\Sigma_2)$ one has
$$
\int_{\Sigma_1\backslash f^{-1}(K_i)} |\nabla(f^*u)|_{f^*|q_2|}^2\,\mathrm{dV}({f^*|q_2|}) = 
\int_{\Sigma_2\backslash K_i} |\nabla u|^2_{|q_2|}\,\mathrm{dV}({|q_2|}). 
$$
Combining this with inequality~\eqref{forms}, one obtains
\begin{align*}
\frac{1}{K}\int_{\Sigma_1\backslash f^{-1}(K_i)} |\nabla(f^*u)|_{f^*|q_1|}^2\,\mathrm{dV}({f^*|q_1|})\leq&\int_{\Sigma_2\backslash K_i} |\nabla u|^2_{|q_2|}\,\mathrm{dV}({|q_2|})\\
\leq& K\int_{\Sigma_1\backslash f^{-1}(K_i)} |\nabla(f^*u)|_{f^*|q_1|}^2\,\mathrm{dV}({f^*|q_1|})
\end{align*}
Passing to the limit $i\to\infty$ and using conformal invariance of the Dirichlet energy leads to
\begin{equation}
\frac{1}{K}\int_{\Sigma_1} |\nabla(f^*u)|_{g_1}^2\,\mathrm{dV}({g_1})\leq \int_{\Sigma_2} |\nabla u|^2_{g_2}\,\mathrm{dV}({g_2})\leq K\int_{\Sigma_1} |\nabla(f^*u)|_{g_1}^2\,\mathrm{dV}({g_1}).
\end{equation}

Let $h_2\in[g_2]$ be a smooth metric on $\Sigma_2$. Then $\mu = (f^{-1})_*v_{h_2}$ defines a measure on $\Sigma_1$ such that 
$$
\int_{\Sigma_1}f^*u\,d\mu = \int_{\Sigma_2}u\,\mathrm{dV}({h_2}).
$$
In particular, $\vol(\Sigma_1,\mu) = \vol(\Sigma_2, h_2)$.

Putting these bounds together, we obtain that
$$
\frac{1}{K}R_{[g_1]}(f^*u,\mu)\leq R_{[h_2]}(u)\leq KR_{[g_1]}(f^*u,\mu).
$$

However, since the Teichm\"uller mapping $f$ is not smooth at zeroes of $q_1$, the measure $\mu$ is not a volume measure of a smooth Riemannian metric. In the last step of this proof we show 
that there exists a sequence of metrics $\rho_n\in[g_1]$ such that $\lambda_1(\rho_n)\to\lambda_1([g_1],\mu)$. 
In order to do that we first obtain a local expression for $\mu$ close to the singular points.

Let $s$ be a zero of $q_1$ and let $z_i$, $i = 1,2$ be local complex coordinates in the neighborhood of $s$ and $f(s)$ respectively such that $q_i = z_i^kdz_i^2$. In cones with vertices at $s$ and $f(s)$ respectively one can introduce the coordinates $\zeta_i = \frac{k+2}{2} \int_0^{z_i}\sqrt{q_i} = z_i^{\frac{k+2}{2}}$. Then in coordinates $\zeta_i$ the mapping $f$ is linear, i.e. in appropriately chosen cones the mapping $f$ in $z_i$-coordinates takes form 
\begin{equation}
\label{teichsing}
f (z_1)= \left(\tilde f (z_1^\frac{k+2}{2})\right)^\frac{2}{k+2},
\end{equation}
where $\tilde f$ is linear and the branch of the root function is chosen so that in the coordinate cone $z_i = \zeta_i^{\frac{2}{k+2}}$.

Now suppose that $\mathrm{dV}({h_2}) = \alpha(z)dz_2d\bar z_2$. Then using~\eqref{teichsing} one obtains
 \begin{align*}
 d\mu =& \alpha(f(z_1)) df(z_1)d\overline{f(z_1)}\\ =&\alpha(f(z_1)) \left|z_1^{-\frac{k+2}{2}}\tilde f(z_1^\frac{k+2}{2})\right|^{-\frac{k}{k+2}}\left(|\tilde f_z(z_1^\frac{k+2}{2})|^2 - |\tilde f_{\bar z}(z_1^\frac{k+2}{2})|^2 \right)dz_1d\bar z_1.
 \end{align*}
 As $\tilde f$ is linear, we conclude that $d\mu = \beta dz_1d\bar z_1$ where $\beta\in L^\infty(\Sigma)$.
At this point an appropriate approximation can be constructed using the following lemma and a standard regularization procedure.

\begin{lemma}
\label{lem:convergence}
Let $g$ be a Riemannian metric on a surface $\Sigma$. Suppose that $\{\rho_\varepsilon\}\subset L^\infty(\Sigma)$ is an equibounded sequence such that $\rho_\varepsilon\to\rho$ as $\varepsilon\to 0$ $\mathrm{dV}(g)$-a.e. Then one has for every $k>0$
$$
\lambda_k([g],\rho_\varepsilon \mathrm{dV}(g)) \to \lambda_k([g],\rho \mathrm{dV}(g)).
$$ 
\end{lemma}

Therefore, taking the supremum over all $h_2$ yields 
$$
\Lambda_k(\Sigma,[g_2])\leq K\Lambda_k(\Sigma,[g_1]).
$$
Switching the role of $\Sigma_1$ and $\Sigma_2$ and considering $f^{-1}$ instead of $f$ in the previous argument completes the proof in the orientable case.

The proof in the non-orientable case is easily reduced to the orientable case using the following construction. For any metric $g$ on $\Sigma$ the metric $\pi^*g$ on $\hat\Sigma$ is $\sigma$-invariant. Thus, its eigenvalues are split into $\sigma$-even and $\sigma$-odd. Moreover, $\sigma$-even eigenvalues coincide with eigenvalues of $(\Sigma,g)$. Since the Teichm\"uller map in this case preserves $\sigma$-even functions, one can repeat the proof of the orientable case, restricting oneself to even eigenvalues. This completes the proof, modulo the proof of the lemma.
\end{proof}

\begin{proof}[Proof of Lemma \ref{lem:convergence}]
The proof of this lemma follows the proof of a similar statement for Steklov eigenvalues found in Lemma 3.1 of \cite{karpukhin_maximal_metrics}. For completeness, we provide the proof. First, we observe there is a constant $C>0$ that does not depend on $\epsilon$ such that $\lambda_k([g], \rho_\epsilon \mathrm{dV}(g))\leq Ck$. This follows from the Theorem $\mathrm{A}_k$ on the top of page 7 of \cite{kokarev}. Moreover, by Proposition 1.1 of \cite{kokarev} we also have 
\begin{align*}
\lim \sup \lambda_k([g], \rho_\epsilon \mathrm{dV}(g)) \leq \lambda_k([g], \rho \mathrm{dV}(g)).
\end{align*} 
Thus, it suffices to prove that $\lambda_k([g], \rho \mathrm{dV}(g)) \leq \lim \inf \lambda_k([g], \rho_\epsilon \mathrm{dV}(g))$. Let $u_\epsilon$ be an eigenfunction corresponding to $\lambda_k([g], \rho_\epsilon \mathrm{dV}(g))$ normalized so that $\|u_\epsilon \|_{L^2(\rho_\epsilon \mathrm{dV}(g))}=1$. We will show that the $L^2(\mathrm{dV}(g))$ and $H^1(\Sigma,\mathrm{dV}(g))$-norms of the $u_\epsilon$ are bounded uniformly in $\epsilon$, for $\epsilon>0$ sufficiently small. We recall the following proposition:
 
 \begin{proposition}{(\cite{MR1411441} Lemma 8.3.1)} 
 \label{prop:bound}
 Let $(M,g)$ be a Riemannian manifold. Then there exists a constant $C>0$ such that for all $L \in H^{-1}(M)$ with $L(1) = 1$ one has
 \begin{align*}
 \|u- L(u) \|_{L^2(M)} \leq C \| L \|_{H^{-1}(M)} \left( \int_M | \nabla u |^2_g \, \mathrm{dV}_g \right)^{1/2}
 \end{align*}
 for all $u \in H^1(M)$. 
 \end{proposition}

We will apply Proposition \ref{prop:bound} with $L_\epsilon (u) = \int_{\Sigma} u \rho_\epsilon \, \mathrm{dV}(g)$. First, we compute the norm of $L_\epsilon$. We have:
\begin{align*}
\left| \int_\Sigma u \rho_\epsilon \mathrm{dV}(g) \right| \leq C \int_{\Sigma} |u| \mathrm{dV}(g) \leq C \| u \|_{L^2(\mathrm{dV}(g))} \leq C \| u \|_{H^1(\Sigma,g)},
\end{align*} 
where we used in order the uniform boundedness of $\rho_\epsilon$, the Cauchy-Schwarz inequality, and the compact embedding of $H^1(\Sigma,g)$ into $L^2(\mathrm{dV}(g))$. Thus, the family of operators $L_\epsilon$ are uniformly bounded in $H^{-1}(\Sigma)$. Applying Proposition \ref{prop:bound} to $L_\epsilon$ and $u_\epsilon$ yields:
\begin{align*}
\| u_\epsilon \|_{L^2(\mathrm{dV}(g))} \leq C \left( \int_\Sigma | \nabla u_\epsilon |^2 \mathrm{dV}(g) \right)^{1/2} = C \sqrt{\lambda_k([g],\rho_\epsilon \mathrm{dV}(g))},
\end{align*}
since $L_\epsilon(u_\epsilon)=0$. Thus, we see that $u_\epsilon$ are uniformly bounded in $H^1(\Sigma, \mathrm{dV}(g))$ and $L^2(\mathrm{dV}(g))$. 

We will now show that the family $u_\epsilon$ is uniformly bounded with respect to $\epsilon$. Indeed, each $u_\epsilon$ satisfies $\Delta_g u_\epsilon = \lambda_k([g],\rho_\epsilon \mathrm{dV}(g)) \rho_\epsilon u_\epsilon$ in a weak sense. Since $\Delta_g$ is a second order elliptic differential operator by the Sobolev Embedding Theorem and elliptic regularity we have:
\begin{align*}
\| u_\epsilon \|_\infty \leq C \| u_\epsilon \|_{H^2( \mathrm{dV}(g))} \leq &C(\| u_\epsilon \|_{L^2(\mathrm{dV}(g))}+ \| \lambda_k([g], \rho_\epsilon \mathrm{dV}(g)) \rho_\epsilon u_\epsilon \|_{L^2(\mathrm{dV(g)})})\\
\leq &C(1+ \lambda_k([g], \rho_\epsilon \mathrm{dV}(g))),
\end{align*}  
where the last inequality comes from the fact that the $L^2$-norms of the $u_\epsilon$ and the $L^\infty$-norms of $\rho_\epsilon$ are uniformly bounded. The claim follows since the eigenvalues are uniformly bounded. 

Now we show that if $u_\epsilon$ and $v_\epsilon$ are $\rho_\epsilon \mathrm{dV}(g)$-orthogonal eigenfunctions then:
\begin{align*}
\int_{\Sigma} u_\epsilon^2 \rho \, \mathrm{dV}(g) \to 1 \hspace{4 pt} \text{ and } \hspace{4 pt} \int_{\Sigma} u_\epsilon v_\epsilon \rho \, \mathrm{dV}(g) \to 0. 
\end{align*}
Since $\| u_\epsilon \|_{L^2(\rho_\epsilon \mathrm{dV}(g))} = 1$ we have:
\begin{align*}
\left| \int_\Sigma u_\epsilon^2 \rho \, \mathrm{dV}(g) -1 \right| \leq \int_{\Sigma} |u_\epsilon|^2 |\rho - \rho_\epsilon | \, \mathrm{dV}(g).
\end{align*}
Since the $u_\epsilon$ are uniformly bounded, the first claim follows. Since $\int_\Sigma u_\epsilon v_\epsilon \rho_\epsilon \, \mathrm{dV}(g) = 0$ a similar argument shows that $\int_{\Sigma} u_\epsilon v_\epsilon \rho \, \mathrm{dV}(g) \to 0$.

Finally, let $E_{k+1}(\epsilon)$ be a direct sum of the first $k$ eigenspaces for $([g], \rho_\epsilon \mathrm{dV}(g))$ with $\rho_\epsilon$-orthonormal basis given by $\left \{ u_\epsilon^i \right \}_{i=1}^{k+1}$. Any function in $E_{k+1}(\epsilon)$ can be written as $\sum_{i=1}^{k+1} c_i u_\epsilon^i$. Plugging this into the Rayleigh quotient yields:
\begin{align*}
\lambda_k([g], \rho \mathrm{dV}(g)) \leq &\frac{\int_{\Sigma} |\sum_{i}c_i\nabla u^i_\epsilon|^2 \, \mathrm{dV}(g)}{\int_\Sigma (\sum_i c_i u^i_\epsilon)^2 \rho \, \mathrm{dV}(g)}\\
= & C_\epsilon \frac{\sum_{i} c_i^2 \int_{\Sigma} |\nabla u_\epsilon^i |^2 \, \mathrm{dV}(g)}{\sum_i c_i^2 \int_\Sigma (u_\epsilon^i)^2 \rho \, \mathrm{dV}(g)} \\
\leq & C_{\epsilon}\max_i \frac{\int_{\Sigma} |\nabla u_\epsilon^i |^2 \, \mathrm{dV}(g)}{\int_\Sigma (u_\epsilon^i)^2 \rho \, \mathrm{dV}(g)},
\end{align*}  
where
\begin{align*}
C_\epsilon = \frac{\sum_i c_i^2\int_\Sigma (u_\epsilon^i)^2 \rho \, \mathrm{dV}(g)}{\sum_i c_i^2\int_\Sigma (u_\epsilon^i)^2 \rho \, \mathrm{dV}(g)+ \sum_{i<j} 2c_i c_j \int_\Sigma u_\epsilon^i u_\epsilon^j \rho \, \mathrm{dV}(g)}
\end{align*}
and in the last inequality we made use of the inequality $\frac{x_1+x_2}{y_1+y_2} \leq \max \left(\frac{x_1}{y_1}, \frac{x_2}{y_2} \right)$, for positive real numbers $x_1,x_2, y_1$ and $y_2$. By our previous observations $C_\epsilon \to 1$ while the numerator of the last term in the inequality is $\lambda_k([g], \rho_\epsilon \mathrm{dV}(g))$ and the denominator goes to one. Thus, we have $\lambda_k([g],\rho \mathrm{dV}(g)) \leq \lim \inf \lambda_k([g], \rho_\epsilon \mathrm{dV}(g))$, which completes the proof.
\end{proof}

\medskip

\subsection{Proof of Theorem \ref{TSmooth}}\label{Proof2}
The proof contains two short steps. First, we prove that the values $\Lambda_1(\mathbb{T}^2)$ and $\Lambda_{1}(\mathbb{KL})$ are achieved by metrics smooth away from finitely many conical singularities. Second, we apply Theorem \ref{unique} to prove that these metrics cannot have conical points, i.e. they must be smooth everywhere.

\begin{proof}

{\bf Step 1.}  
As we discussed in Sections~\ref{Conf&Max torus} and~\ref{Conf&Max} the space of conformal classes $\mathbb{T}^2$ and $\mathbb{KL}$ can be identified with subsets of $\mathbb{R}$ and $\mathbb{C}$ respectively. Moreover, the induced  topology coincides with the topology generated by Tecihm\"uller distance, see~\cite{MR2850125}.

 Let $\Sigma$ denote either $\mathbb{T}^2$ or $\mathbb{KL}$ and $\{g_n\}$ be a sequence of metrics on $\Sigma$ such that $\lim\bar\lambda_1(\Sigma,g_n)\to \Lambda_1(\Sigma)$.
From \cite{MR1415764} and \cite{MR2209284} we know that:
 $$
 \Lambda_1(\mathbb{T}^2) > 8\pi
 $$
and
 $$
 \Lambda_{1}(\mathbb{KL}) \geq  12\pi E(2\sqrt{2}/3) > 12\pi.
 $$
 Therefore, by Theorems \ref{Girouard torus}, \ref{Girouard} the conformal classes $[g_n]$ belong to a compact subset of the space of conformal classes. Thus, the sequence $\{[g_n]\}$ has a limit point $[g]$. By Proposition \ref{cont} one has $\Lambda_{1}(\Sigma,[g]) =\Lambda_1(\Sigma)$.
It was proved by Petrides \cite{petrides} that for any conformal class $[h]$ there exists a metric $\widetilde{h}\in [h]$, smooth except maybe at a finite number of singular points corresponding to conical singularities, such that $\Lambda_{1}(\Sigma,[h])=\bar\lambda_{1}(\Sigma, \widetilde{h})$. We conclude that there exists a metric $\widetilde g\in[g]$ such that $\Lambda_1(\Sigma) = \bar\lambda_1(\Sigma,\widetilde g)$. Moreover, by Theorem \ref{extremal} and Remark \ref{Klein} $\widetilde{g}$ is induced from a (possibly branched) minimal immersion by first eigenfunctions $\Phi$ of $\Sigma$ into a round sphere of dimension at least three.

{\bf Step 2.} Suppose that $\widetilde g$ has a conical point. From Theorem \ref{extremal} it follows that this metric is induced from a branched minimal isometric immersion into a round sphere.
 %which is equivalent to the condition that this immersion is a harmonic map 
 %(see, \cite[Example p. 9]{MR1363514}) 
 As it was observed in sections \ref{Conf&Max torus} and \ref{Conf&Max}, on $\Sigma$ there exist natural $\mathbb{S}^1$-actions $s_\theta$ by conformal diffeomorphisms without fixed points. Then the mapping $\Phi\circ s_\theta$ is again a branched minimal immersion. The metric induced by this immersion is $s_\theta^*\widetilde g$. Therefore, since the components of $\Phi$ are the first eigenfunctions of $(\Sigma,\widetilde g)$, then the components of $s_\theta^*\Phi = \Phi\circ s_\theta$ are the first eigenfunctions of $(\Sigma,s_\theta^*\widetilde g)$.
%  Let $\widetilde{g}$ be a metric on $\Sigma$ obtained by this action. Then the composition of this conformal diffeomorphism and the immersion is a harmonic map into the same sphere (see, for instance,  \cite[subsection (10.2)]{MR1363514}) and $\widetilde{g}$ is induced from this map. Therefore, we obtain a branched minimal isometric immersion of  $(\Sigma,\widetilde{g})$ into the sphere. Moreover, this immersion is given by first eigenfunctions as well (with respect to the induced metric $\widetilde{g}$) because the action is by conformal diffeomorphisms. 
  By Theorem \ref{unique} the metrics $s_\theta^*\widetilde{g}$ and $\widetilde{g}$ must be equal. Thus, a $\bar\lambda_1$-maximal metric must be a metric of revolution. Under this $\mathbb{S}^1$-action the conical point forms a 1-dimensional singular set, which contradicts Step 1 (the set of conical points of $\Lambda_1$-maximal metric is at most finite). This completes the proof of the theorem. 
\end{proof}

\bibliography{mybib}
\bibliographystyle{alpha}

\end{document}